\documentclass[a4paper,oneside,11pt]{amsart}
\usepackage{graphicx, geometry}
\usepackage{latexsym,amssymb,amsmath,amscd,amsthm,amsxtra}
\usepackage{amsfonts}
\usepackage{bbm}
\usepackage{mathrsfs}
\usepackage[all]{xy}

\def\lprod{{>\!\!\!\triangleleft\kern-.33em\ \, }}

\def\C{\mathbb{C}}

\def\g{{\mathfrak{g}}}

\def\gl{{\mathfrak{gl}}}
\def\A{{\cal A}}

\def\L{{\cal L}}
\def\E{{\cal E}}
\def\M{{\cal M}}

\def\C{{\cal C}}
\def\P{{\cal P}}
\def\V{{\cal V}}

\def\cal#1{\mathcal{#1}}

\newcommand{\id}{\operatorname{id}}

\newcommand{\Der}{\operatorname{Der}}

\newtheorem{theorem}{Theorem}[section]
\newtheorem{proposition}[theorem]{Proposition}
\newtheorem{definition}[theorem]{Definition}

\newtheorem{lemma}[theorem]{Lemma}
\newtheorem{example}[theorem]{Example}

\newcommand{\pf}{\noindent{{\bf Proof.}\ \ }}
\def\qed{\hfill {$\Box$} \\[-.5em]}

\allowdisplaybreaks

\begin{document}

\title{Cohomology and crossed modules extension of Hom-Leibniz-Rinehart algebras}

\author[Y. Bi]{Yanhui Bi}
\address{Center for Mathematical Sciences, College of Mathematics and Information Science\\
Nanchang Hangkong University\\
Jiangxi 330063, PR China}
\email{biyanhui0523@163.com}

\author[D. Chen]{Danlu Chen}
\address{College of Mathematics and Information Science\\
Nanchang Hangkong University\\
Jiangxi 330063, PR China}
\email{danluchen1008@163.com}

\author[T. Zhang]{Tao Zhang}
\address{College of Mathematics and Information Science\\
Henan Normal University\\
Xinxiang 453007, PR China}
\email{zhangtao@htu.edu.cn}
\thanks{Corresponding author: zhangtao@htu.edu.cn}

\date{}

\begin{abstract}
In this paper, we introduce the concept of crossed module for Hom-Leibniz-Rinehart algebras.
We study the cohomology and extension theory of Hom-Leibniz-Rinehart algebras.
It is proved that there is one-to-one correspondence between equivalence classes of abelian extensions of Hom-Leibniz-Rinehart algebras and the elements of second cohomology group.
Furthermore, we prove that there is a natural map from $\alpha$-crossed modules extension of Hom-Leibniz-Rinehart algebras to the third cohomology group of Hom-Leibniz-Rinehart algebras.
\end{abstract}



\maketitle

\section{Introduction}

Leibniz algebras are fundamental structures in mathematics and physics. The notion of Leibniz algebras was first proposed by Blokh \cite{Blo} in 1965 under the name $D$-algebras as a natural generalization of Lie algebras. Leibniz algebras later were rediscovered by Loday \cite{Lod1} who called them Leibniz algebras as noncommutative analogues of Lie algebras. In the past two decades the theory of Leibniz algebras has been extensively studied
and many results of Lie algebras have been extended to Leibniz algebras, such as, classical results on Cartan subalgebras \cite{Cas,Alb1,Omir1}. Up until to now the Levi's theorem for Leibniz algebras \cite{Bar}, the representation, homology and cohomology of Leibniz algebras \cite{Lod2,CP} have been studied by many people, see also \cite{Gao1999,Gao2000,Wang,Liu}.

As an algebraic counterpart of Lie algebroid, the concept of Lie-Rinehart algebra was introduced by Herz \cite{Herz}, Palais \cite{Palais} and Rinehart \cite{Rinehart} under the name of $(R,C)$-Lie algebra, differential Lie algebra, Lie pseudoalgebra and so on. See also \cite{Huebschmann,Mackenzie} and the reference therein.
In \cite{CLP}, the notion of crossed modules for Lie-Rinehart algebras was introduced and studied by using the cohomology theory of Lie-Rinehart algebras.
In \cite{Alp}, a new characterization of crossed modules for Lie-Rinehart algebras by using cat${^1}$-Lie-Rinehart algebras was given.

On the other hand, mathematicians have an increasing interest in the studies of crossed module and Hom-structure in recent years.
The notion of Hom-Leibniz algebras is a natural generalization of Leibniz algebras and Hom-Lie algebras, see \cite{Makhlouf2010}.
For other Hom-types of algebras include Hom-associative algebras, Hom-Nambu-Lie algebras, Hom-Hopf algebras, Hom-Poisson algebras, Hom-Lie-Yamaguti algebras, see \cite{Makhlouf,Sheng,Yau}.
Very recently, Hom-Lie algebroids and Hom-Lie-Rinehart algebra were studied in \cite{Camille} and \cite{Mandal}.
Crossed modules for Hom-Lie-Rinehart algebras were studied by the first and third author in \cite{Zhang1}.
Categorical properties of  crossed module of algebroid were studied by M. Alp in \cite{Alp1,Alp2}.
Crossed modules for Hom-Lie antialgebras were studied in \cite{Zhang2}.

In this paper, we introduce the concept of crossed modules for Hom-Leibniz-Rinehart algebras and cat${^1}$-Hom-Leibniz-Rinehart algebras.
We give a detailed study on its construction from the Hom-actions and semi-direct products of Hom-Leibniz-Rinehart algebras.
It is proved that there is a one-to-one correspondence between crossed modules of Hom-Leibniz-Rinehart algebras and cat$^1$-Hom-Leibniz-Rinehart algebras.
Finally, we study the cohomology and crossed module extension theory of Hom-Leibniz-Rinehart algebras.

The paper is organized as follows. In section 2,  we review some basic notions about Hom-Leibniz algebras and Hom-Leibniz-Rinehart algebras.
In Section 3, we introduce the notion of crossed modules for Hom-Leibniz-Rinehart algebras and construct Hom-Leibniz-Rinehart algebras from the Hom-actions and semi-direct products.
In Section 4, we define the concept of cat$^1$-Hom-Leibniz-Rinehart algebras and give its relationship with crossed modules of Hom-Leibniz-Rinehart algebras.
In Section 5,  we study the cohomology theory of Hom-Leibniz-Rinehart algebras. It is proved that there is one-to-one correspondence between equivalence classes of abelian extensions of Hom-Leibniz-Rinehart algebras and the elements of second cohomology group.
In Section 6,  the relationship between the third cohomology group and the $\alpha$-crossed module extensions theory of Hom-Leibniz-Rinehart algebras  are investigated.

\section{Preliminaries}
Let us recall some terminology and notions used in the paper.
We mainly follow \cite{Casas,Hartwig,Makhlouf,Mandal,Sheng,Yau,Makhlouf2010}.

\begin{definition} \cite{Makhlouf2010} 
A Hom-Leibniz algebra is a triple $(\L, [\cdot,\cdot]_\L, \alpha_\L)$ where $\L$ a vector space equipped with
a bilinear map $[\cdot,\cdot]_\L:\L\otimes \L\to\L$  and a linear map $\alpha_\L:\L\to \L$ satisfying the following
hom-Jacobi identity:
\begin{equation}
[\alpha_\L(x),[y,z]_\L]_\L=[[x,y]_\L,\alpha_\L(z)]_\L+[\alpha_\L(y),[x,z]_\L]_\L,
\end{equation}
for all $x,y,z\in\L$.
A Hom-Leibniz algebra is called multiplicative if $\alpha$ is an algebraic homomorphism, i.e. $\alpha([x,y]_\L)= [\alpha(x), \alpha(y)]_\L.$
A Hom-Leibniz algebra is called regularif $\alpha$ is an algebraic automorphism.

A homomorphism between two Hom-Leibniz algebras $(\L,[\cdot,\cdot]_\L,\alpha_\L)$ and $(\L',[\cdot,\cdot]_{\L'},\alpha_{\L'})$ is a linear map $\phi:\L\to\L'$ such that
\begin{equation}
\phi(\alpha_\L(x))=\alpha_{\L'}(\phi(x)),\ \ \ \phi([x,y]_\L)=[\phi(x),\phi(y)]_{\L'},
\end{equation}
for all $x,y\in\L.$
\end{definition}

This is in fact the definition of a left Hom-Leibniz algebra. The dual notion of right Hom-Leibniz algebra is made out of the  relation $[[x, y]_\L, \alpha_\L(z)] = [\alpha_\L(x), [y, z]_\L]_\L - [\alpha_\L(y), [x, z]_\L]_\L$, for all $x, y, z \in \L$.
In this paper, we are mainly considering left Hom-Leibniz algebras and all Hom-Leibniz algebras are assumed to be multiplicative and regular.


\begin{definition} 
 A representation of a Hom-Leibniz algebra $(\L, [\cdot,\cdot]_\L, \alpha_\L)$ on a Hom-vector space $(V,\alpha_V)$ is a pair $(\rho^{L} ,\rho^{R})$ where $\rho^{L} , \rho^{R}:\L\to \gl(V)$  such that the following conditions hold: $\forall x,y\in\L$,
\begin{eqnarray}
\rho^{L} (\alpha_\L(x))\circ\alpha_V&=&\alpha_V\circ\rho^{L} ,\\
\rho^{R}(\alpha_\L(x))\circ\alpha_V&=&\alpha_V\circ\rho^{R},\\
\rho^{L} ([x,y])\circ\alpha_V&=&\rho^{L} (\alpha_\L(x))\circ\rho^{L} (y)-\rho^{L} (\alpha_\L(y))\circ\rho^{L} (x),\\
\rho^{R}([x,y])\circ\alpha_V&=&\rho^{L} (\alpha_\L(x))\circ\rho^{R}(y)-\rho^{R}(\alpha_\L(y))\circ\rho^{L} (x),\\
\rho^{R}(\alpha_\L(y))\circ\rho^{R}(x)&=&-\rho^{R}(\alpha_\L(y))\circ\rho^{L} (x).
\end{eqnarray}
\end{definition}

\begin{definition}\cite{Casas} 
Let $({\L},\alpha_{\L})$ and $({\M},\alpha_{\M})$ be Hom-Leibniz algebras. A Hom-action of $({\L},\alpha_{\L})$ on $({\M},\alpha_{\M})$ is a pair of linear maps ${\L}\otimes {\M}\to {\M}, (x, m)\mapsto [x,m]$ and ${\M}\otimes {\L}\to {\M}$, $(m, x)\mapsto [m,x],$ satisfying the following properties:
\begin{enumerate}
\item[(A11)] $[\alpha_{\M}(m),[x,y]]=[\alpha_{\L}(x),[m,y]]-[[x,m],\alpha_{\L}(y)],$
\item[(A12)] $[\alpha_{\L}(x),[m,y]]=[\alpha_{\M}(m),[x,y]]-[[m,x],\alpha_{\L}(y)],$
\item[(A13)] $[\alpha_{\L}(x),[y,m]]=[[x,y],\alpha_{\M}(m)]+[\alpha_{\L}(y),[x,m]],$
\item[(A21)] $[\alpha_{\L}(x),[m,m']]=[\alpha_{\M}(m),[x,m']]-[[m,x],\alpha_{\M}(m')],$
\item[(A22)] $[\alpha_{\M}(m),[x,m']]=[\alpha_{\L}(x),[m,m']]-[[x,m],\alpha_{\M}(m')],$
\item[(A23)] $[\alpha_{\M}(m),[m',x]]=[[m,m'],\alpha_{\L}(x)]+[\alpha_{\M}(m'),[m,x]],$
\item[(A31)] $\alpha_{\M}([x,m])=[{\alpha_{\L}(x)},\alpha_{\M}(m)],$
\item[(A32)] $\alpha_{\M}([m,x])=[\alpha_{\M}(m),{\alpha_{\L}(x)}].$
\end{enumerate}
\noindent for all $x,y\in {\L}$ and $m,m'\in {\M}.$
The Hom-action is called trivial if $[x,m]=0$ and $[m,x]=0$ for all $x\in {\L}$ and $m\in {\M}.$
\end{definition}

\begin{definition} 
Given an algebraic homomorphism $\varphi$ of an associative commutative algebra $\A$, we call $\delta:\A\to \A$ a $\varphi$-derivation if it satisfies
\begin{equation}
\delta(fg)=\varphi(f)\delta(g)+\varphi(g)\delta(f)
\end{equation}
for all $f,g\in \A.$
\end{definition}

Let $\A$ be an associative commutative algebra over a field $\mathbb{K}$,
we denote $\Der_{\varphi}\A$ the set of  all the $\varphi$-derivations of $\A$.

\begin{definition}
A Hom-Leibniz-Rinehart algebra over $(\A,\varphi)$ is a tuple $(\L, [\cdot, \cdot]_\L, \alpha_\L, \rho^{L}_\L,\rho^{R}_{\L})$, where $\A$ is an associative commutative algebra, $\varphi:\A\to \A$ is an algebraic homomorphism,
$\L$ is an $\A$-module, $(\L, [\cdot, \cdot]_\L, \alpha_\L)$ is a Hom-Leibniz algebra, together with the maps $\rho^{L}_\L,\rho^{R}_{\L}:\L\to \Der_{\varphi}\A$ (called left and right anchor), such that the  following compatible conditions are satisfied:
\begin{enumerate}
\item[(H01)] $\alpha_\L(fx)=\varphi(f)\alpha_\L(x),$
\item[(H11)] $\rho^{L}_\L(\alpha_\L(x))\circ\varphi(f)=\varphi\circ\rho^{L}_\L(x)(f),$
\item[(H12)] $\rho^{R}_{\L}(\alpha_\L(x))\circ\varphi(f)=\varphi\circ\rho^{R}_{\L}(x)(f),$
\item[(H21)] $\rho^{R}_{\L}(\alpha_\L(y))\circ\rho^{R}_{\L}(x)(f)=-\rho^{R}_{\L}(\alpha_\L(y))\circ\rho^{L}_\L(x)(f),$
\item[(H22)] $\rho^{L}_\L([x,y])\circ\varphi(f)=\rho^{L}_\L(\alpha_\L(x))\circ\rho^{L}_\L(y)(f)-\rho^{L}_\L(\alpha_\L(y))\circ\rho^{L}_\L(x)(f),$
\item[(H23)] $\rho^{R}_{\L}([x,y])\circ\varphi(f)=\rho^{L}_\L(\alpha_\L(x))\circ\rho^{R}_{\L}(y)(f)-\rho^{R}_{\L}(\alpha_\L(y))\circ\rho^{L}_\L(x)(f),$
\item[(H31)] $[x, fy]=\varphi(f)[x, y]+\rho^{L}_\L(x)(f)\alpha_\L(y),$
\item[(H32)] $[fx, y]=\varphi(f)[x, y]-\rho^{R}_{\L}(y)(f)\alpha_\L(x),$
\item[(H41)] $\rho^{L}_\L(fx)=\varphi(f)\rho^{L}_\L(x),$
\item[(H42)] $\rho^{R}_{\L}(fx)=\varphi(f)\rho^{R}_{\L}(x),$
\end{enumerate}
where $f\in \A$ and $x, y\in \L$.
We will denote a Hom-Leibniz-Rinehart algebra  by $(\L, [\cdot, \cdot]_\L, \alpha_\L, \rho^{L}_\L,\rho^{R}_{\L})$ or simply $\L$.
\end{definition}
Note that from (H11)-(H23) we know that $(\rho^{L}_\L,\rho^{R}_{\L})$ is a representation of Hom-Leibniz algebra $(\L, [\cdot, \cdot]_\L, \alpha_\L)$ over $(\A,\varphi)$.
We remark that our definition of Hom-Leibniz-Rinehart algebra  is different from \cite{Guo} since in that paper the conditions (H12), (H21), (H23), (H32) and (H42) are missed so they only consider the special case $\rho^{R}_{\L}=0$.

\begin{example}
If $\rho^{L}_\L=0=\rho^{R}_{\L}$, then a Hom-Leibniz-Rinehart algebra $(\L, [\cdot, \cdot]_\L, \alpha_\L, \rho^{L}_\L,\rho^{R}_{\L})$ is a Hom-Leibniz $\A$-algebra. If $\phi=\id, \alpha_\L=\id$, then a Hom-Leibniz-Rinehart algebra is reduced to a Leibniz-Rinehart algebra or Loday QD-Rinehart algebra defined in \cite{Cas20}.
\end{example}
\begin{example}
If we consider a Loday QD-Rinehart algebra $\L$ over $\A$ along with an endomorphism
$$
(\alpha_{\L},\varphi ): (\L, \A) \rightarrow(\L, \A)
$$
in the category of Leibniz-Rinehart algebras, then we get a Hom-Leibniz-Rinehart algebra $(\L, \A)$ as follows:
 $$[x, y]_{\L}=\alpha_{\L}([x, y]) ,$$
 $$\rho^{L}_\L(x)(f)=\varphi(\rho^{L}_\L(x)(f)) ,$$
 $$\rho^{R}_{\L}(x)(f)=\varphi(\rho^{R}_{\L}(x)(f)) ,$$
for all $x, y \in \L ,f \in \A$.
\end{example}

\begin{definition}\label{morphism1} 
Let  $(\L, \A)$ and  $(\L', \A)$ be two Hom-Leibniz-Rinehart algebra over $(\A,\varphi)$.
Then a homomorphism $\Phi:\L\to \L'$ is an $\A$-module homomorphism such that the following conditions are satisfied:
\begin{eqnarray}
\Phi\circ\alpha_{\L}&=&\alpha_{\L'}\circ\Phi,\quad \Phi([x,y]_{\L})=[\Phi(x),\Phi(y)]_{\L'},\\
\rho^{L}_{\L'}\circ\Phi&=&\rho^{L}_\L,\quad \rho^{R}_{\L'}\circ\Phi=\rho^{R}_{\L},
\end {eqnarray}
for all $x, y \in \L ,f \in \A$.
\end{definition}
Note that the first two conditions mean that $\Phi:\L\to \L'$ is a Hom-Leibniz algebra homomorphism,
and the last two conditions mean that the anchors of $\L$ are preserved by $\Phi$.


\section{Crossed modules of Hom-Leibniz-Rinehart algebras}
In this section, we introduce the concept of crossed module for Hom-Leibniz-Rinehart algebras.
First let us define the Hom-actions and semi-direct product of Hom-Leibniz-Rinehart algebras.

\begin{definition}\label{semi} 
Let $(\L,[\cdot,\cdot]_\L, \alpha_\L,\rho^{L}_\L,\rho^{R}_{\L})$ be a Hom-Leibniz-Rinehart algebra
and $(\M,[\cdot,\cdot]_\M,\alpha_\M)$ be a Hom-Leibniz algebra which is also an $\A$-module  (called Hom-Leibniz $\A$-algebra).
We will say that $\L$ acts on $\M$ if there is a pair of $\mathbb{K}$-linear maps
$$\L\otimes\M\to\M,\ \ \ (x,m)\mapsto[x,m],\ \ \ x\in\L,m\in\M,$$
$$\M\otimes\L\to\M,\ \ \ (m,x)\mapsto[m,x],\ \ \ x\in\L,m\in\M,$$
such that the following identities hold:
\begin{enumerate}
\item[(S11)] $\L$ acts on $\M$ as a Hom-Leibniz algebra,
\item[(S21)] $[fx,m]=\varphi(f)[x,m],$
\item[(S22)] $[m,fx]=\varphi(f)[m,x],$
\item[(S31)] $[x,fm]=\varphi(f)[x,m]+\rho_\L^{L}(x)(f)\alpha_\M(m),$
\item[(S32)] $[fm,x]=\varphi(f)[m,x]-\rho_\L^{R}(x)(f)\alpha_\M(m).$
\end{enumerate}
\end{definition}
Now for a Hom-Leibniz-Rinehart algebra $(\L,[\cdot,\cdot]_\L, \alpha_\L,\rho^{L}_\L,\rho^{R}_{\L})$ and a Hom-Leibniz $\A$-algebra $(\M,[\cdot,\cdot]_\M,\alpha_\M)$ where $\L$ acts on it. We consider direct sum $\L\oplus\M$ as an $\A$-module: $f(x,m)=(fx,fm).$
Define the anchors $\rho^{L}_\L,\rho^{R}_{\L}:\L\oplus\M\to \Der_\varphi(\A)$:
\begin{equation}
\rho^{L}_\L(x,m)=\rho^{L}_\L(x),\quad\rho^{R}_{\L}(x,m)=\rho^{R}_{\L}(x),
\end{equation}
the map $\alpha:\L\oplus\M\to\L\oplus\M$:
\begin{equation}
\alpha(x,m)=(\alpha_\L(x),\alpha_\M(m)),
\end{equation}
and the bracket:
\begin{equation}
[(x,m),(x',m')]_{\L\rtimes\M}=([x,x']_\L,[m,m']_\M+[x,m']+[m, x']),
\end{equation}
for all $m,m' \in \M, x,x' \in \L $.

\begin{lemma}\label{TH2}
The vector space $\L\oplus\M$ equipped with the above maps and bracket is a  Hom-Leibniz-Rinehart algebra if and only if the conditions (S11)-(S32) hold.
This is called a semi-direct product of Hom-Leibniz-Rinehart algebras of $\L$ and $\M$, denoted by $\L\rtimes\M$.
\end{lemma}

\begin{proof}
By direct calculations, we will only verify the conditions (H01), (H11), (H21), (H31) and (H41) hold.
For conditions (H01) and (H11), we have
\begin{eqnarray*}
\alpha(f(x,m))
&=&\alpha(fx,fm)\\
&=&(\alpha_\L(fx),\alpha_\M(fm))\\
&=&(\varphi(f)\alpha_\L(x),\varphi(f)\alpha_\M(m))\\
&=&\varphi(f)(\alpha_\L(x),\alpha_\M(m))\\
&=&\varphi(f)\alpha(x,m),
\end{eqnarray*}
\begin{eqnarray*}
\rho^{L}_\L(\alpha(x,m))\varphi(f)
&=&\rho^{L}_\L(\alpha_\L(x),\alpha_\M(m))\varphi(f)\\
&=&\rho^{L}_\L(\alpha_\L(x))\varphi(f)\\
&=&\varphi(\rho^{L}_\L(x)(f))\\
&=&\varphi(\rho^{L}_\L(x,m)(f)),
\end{eqnarray*}
For condition (H21), we have
\begin{eqnarray*}
&&\rho^{L}_\L([(x,m),(x',m')])\varphi(f)\\
&=&\rho^{L}_\L([x,x']_\L, [m,m']_\M+[x,m']+[m,x')\varphi(f)\\
&=&\rho^{L}_\L([x,x']_\L)\varphi(f)\\
&=&\rho^{L}_\L(\alpha_\L(x))(\rho^{L}_\L(x')(f))-\rho^{L}_\L(\alpha_\L(x'))(\rho^{L}_\L(x)(f))\\
&=&\rho^{L}_\L(\alpha_\L(x),\alpha_\M(m))(\rho^{L}_\L(x',m')(f))-\rho^{L}_\L(\alpha_\L(x'),\alpha_\M(m'))(\rho^{L}_\L(x,m)(f))\\
&=&\rho^{L}_\L(\alpha(x,m))(\rho^{L}_\L(x',m')(f))-\rho^{L}_\L(\alpha(x',m'))(\rho^{L}_\L(x,m)(f)),
\end{eqnarray*}
For conditions (H31), we get
\begin{eqnarray*}
&&[(x,m),f(x',m')]\\
&=&[(x,m),(fx',fm')]\\
&=&([x,fx']_\L, [m,fm']_\M+[x,fm']+[m,fx'])\\
&=&\big(\varphi(f)[x,x']_\L+\rho^{L}_\L(x)(f)\alpha_\L(x'),\\
&&\ \ \ \varphi(f)[m,m']_\M+\varphi(f)[\alpha_\L(x),m']+\rho^{L}_\L(\alpha_\L(x))(f)\alpha_\M(m')\\
&&\ \ \ +\varphi(f)[m,\alpha_\L(x')]\big)\\
&=&\big(\varphi(f)[x,x']_\L,\varphi(f)[m,m']_\M+\varphi(f)[\alpha_\L(x),m']+\varphi(f)[m,\alpha_\L(x')]\big)\\
&&+\big(\rho^{L}_\L(x)(f)\alpha_\L(x'),\rho^{L}_\L(\alpha_\L(x))(f)\alpha_\M(m')\big)\\
&=&\varphi(f)[(x,m),(x',m')]+\big(\rho^{L}_\L(x)(f)\alpha_\L(x'),\rho^{L}_\L(x)(f)\alpha_\M(m')\big)\\
&=&\varphi(f)[(x,m),(x',m')]+\rho^{L}_\L(x,m)(f)\alpha(x',m').
\end{eqnarray*}
For conditions (H41), we obtain
\begin{eqnarray*}
\rho^{L}_\L(f(x,m))=\rho^{L}_\L(fx,fm)=\rho^{L}_\L(fx)=\varphi(f)\rho^{L}_\L(x)=\varphi(f)\rho^{L}_\L(x,m).
\end{eqnarray*}
Thus $\L\rtimes\M$ is indeed a Hom-Leibniz-Rinehart algebra. The converse is also true, we omit the details.
\end{proof}

Next we define crossed modules for Hom-Leibniz-Rinehart algebras. We will give an equivalent condition using homomorphism between semi-direct product.

\begin{definition}\label{CRO} 
A crossed module $\partial:\M\rightarrow\L$ of Hom-Leibniz-Rinehart algebras over $(\A,\varphi)$ consists of a Hom-Leibniz-Rinehart algebra
$(\L,[\cdot,\cdot]_\L, \alpha_\L,\rho^{L}_\L,\rho^{R}_{\L})$, a Hom-Leibniz $\A$-algebra $(\M,[\cdot,\cdot]_\M,\alpha_\M)$ together with the action of $\L$ on $\M$ such that the following identities hold:
\begin{enumerate}
\item[(CM0)] $\alpha_\L\circ\partial(m)=\partial\circ\alpha_\M(m),$
\item[(CM1)] $\partial[x,m]=[x,\partial m]_\L,\partial[m,x]=[\partial m,x]_\L,$
\item[(CM2)] $[\partial(m), n]=[m,n]_\M=[m,\partial(n)],$
\item[(CM3)] $\partial(fm)=f\partial(m),$
\item[(CM4)] $\rho^{L}_\L(\partial(m))(f)=0=\rho^{R}_{\L}(\partial(m))(f)$,
\end{enumerate}
\noindent where $m, n\in\M,x\in\L,f\in\A.$

\end{definition}

From the first three conditions above, we see that $\partial:\M\rightarrow\L$ is a crossed module of Hom-Leibniz algebras,
the condition (CM3) means that $\partial$ is a homomorphism of $\A$-modules and the condition
(CM4) shows that the composition of the following maps are zero:
$$\M \xrightarrow{\partial} \L \xrightarrow{\rho^{L}_\L} Der(\A),\quad\M \xrightarrow{\partial} \L \xrightarrow{\rho^{R}_{\L}} Der(\A).$$

\begin{theorem}\label{TH2}
Let $(\L,[\cdot,\cdot]_\L, \alpha_\L,\rho^{L}_\L,\rho^{R}_{\L})$ be a Hom-Leibniz-Rinehart algebra
and $(\M,$\ $[\cdot,\cdot]_\M$,$\alpha_\M)$ be a Hom-Leibniz $\A$-algebra.
Then  a homomorphism $\partial:\M\rightarrow\L$ is a crossed module of Hom-Leibniz-Rinehart algebras
if and only if the maps $(\id, \partial) : \L\rtimes\M \to \L\rtimes\L$ and $(\partial,\id) : \M\rtimes\M \to \L\rtimes\M$ are
homomorphisms of Hom-Leibniz-Rinehart algebras.
\end{theorem}

\begin{proof}
To see when the maps $(\id, \partial) : \L\rtimes\M \to \L\rtimes\L$ and $(\partial,\id) : \M\rtimes\M \to \L\rtimes\M$ are
homomorphisms of Hom-Leibniz-Rinehart algebras, we first check that
\begin{eqnarray*}
(\id, \partial)\circ\alpha_{\L\rtimes\M}(x,m)&=&(\id, \partial)(\alpha_\L(x),\alpha_\M(m))\\
&=&(\alpha_\L(x),\partial\alpha_\M(m))=(\alpha_\L(x),\alpha_\L\partial(m))\\
\alpha_{\L\rtimes\L}\circ(\id, \partial)(x,m)&=&\alpha_{\L\rtimes\L}(x,\partial(m))=(\alpha_\L(x),\alpha_\L\partial(m))
\end{eqnarray*}
thus $(\id, \partial)\circ\alpha_{\L\rtimes\M}=\alpha_{\L\rtimes\L}\circ(\id, \partial)$ if and only if $\alpha_\L\circ\partial(m)=\partial\circ\alpha_\M(m)$.
This is condition (CM0).
On the other hand, we calculate
\begin{eqnarray*}
{(\id, \partial) [(x, m), (x', m')]} &=& ([x,x']_\L, \partial( [x, m '])+\partial([m, x'] )+ \partial[m, m']) \\
{[(\id, \partial)(x, m), (\id, \partial)(x', m')]}&
=&([x,x']_\L, [x,  \partial (m')]+[\partial(m),x']
+[\partial(m), \partial(m')])
\end{eqnarray*}
thus the above two formula are equal to each other if and only if $\partial[x,m]=[x,\partial m]_\L$ and $\partial[m,x]=[\partial m,x]_\L$ hold for all $x\in\L, m\in \M$.
\begin{eqnarray*}
{(\partial,\id) [(m, n), (m', n')] } &=& (\partial[m, m'], [m, n'] + [n, m'] + [n, n'] )\\
{[(\partial,\id)(m, n), (\partial,\id) (m', n')]} &=& ([\partial (m), \partial (m')],[\partial(m),n']+[n, \partial(m')]+ [n, n'])
\end{eqnarray*}
thus the above two formula are equal to each other if and only if $[\partial (m), n]=[m,n]_\M=[ m, \partial(n)]$ hold for all $m, n\in \M$.
From above we obtained conditions (CM1) and (CM2).

Since the maps $(\id, \partial) : \L\rtimes\M \to \L\rtimes\L$ and $(\partial,\id) : \M\rtimes\M \to \L\rtimes\M$ are
homomorphisms of $\A$-modules, we get condition (CM3).
Note that the left anchor of $\L\rtimes\M$ and $\L\rtimes\L$ are given by $\rho^{L}_\L(x,m)=\rho^{L}_\L(x)$ and $\rho^{L}_\L(x,y)=\rho^{L}_\L(x)$ respectively.
Then we obtain
$$\rho^{L}_\L((\partial,\id) (m, n))=\rho^{L}_\L(\partial(m),n)=\rho^{L}_\L(\partial(m))=\rho^{L}_\L(m,n).$$
Thus the left anchor of $\M\rtimes\M$ is preserved by $(\partial,\id)$ if and only if $\rho^{L}_\L(\partial(m))(f)=0$ for all $f\in \A$ since the left anchor of $\M\rtimes\M$ is zero.
 Similarly, we can prove the case of the right anchor. This is condition (CM4).
\end{proof}

\begin{definition}\label{morphism01} 
Let  $\partial:\M\rightarrow\L$ and $\partial':\M'\rightarrow\L'$ be two crossed modules of Hom-Leibniz-Rinehart algebras over $(\A,\varphi)$.
Then a homomorphism between them is a pair $(\Phi,\Psi)$ of morphism $\Phi:\M\rightarrow\M'$  and  $\Psi:\L\rightarrow\L'$
satisfying the following additional conditions:

$$\Psi\circ\partial=\partial'\circ\Phi,\quad \Phi([x,m])=[\Psi(x),\Phi(m)],\quad\Phi([m,x])=[\Phi(m),\Psi(x)].$$
\end{definition}

The category of crossed modules of Hom-Leibniz-Rinehart algebras will be denoted by $\mathbf{CM}$.

\section{Cat$^1$-Hom-Leibniz-Rinehart algebra}

In this section, we introduce the concept of Cat$^1$-Hom-Leibniz-Rinehart algebras.
We prove that there is a category equivalence
between the category of crossed modules of Hom-Leibniz-Rinehart algebras and the category of
cat$^1$-Hom-Leibniz-Rinehart algebras.

\begin{definition}\label{DEF22} 
A cat$^1$-Hom-Leibniz-Rinehart algebra $\C=(i;s,t:\P\to\L)$ has source Hom-Leibniz-Rinehart algebra $(\P,[\cdot,\cdot]_\P,\alpha_\P)$, range Hom-Leibniz-Rinehart algebra $(\L,[\cdot,\cdot]_\L,\alpha_\L)$, and two homomorphisms $s,t:\P\to\L$ and an embedding $i:\L\to\P$ satisfying the following conditions:
\begin{enumerate}
\item[(Cat1)] $s\circ i\circ t=t\circ\alpha,\  t\circ i\circ s=s\circ\alpha,$
\item[(Cat2)] $[\ker s,\ker t]=[\ker t,\ker s]=0,$
\item[(Cat3)] $t\circ\xi(fx)=\varphi(f)t\circ\xi(x),$
\item[(Cat4)] $\rho^L\circ \xi\circ s(x)(f)=0=\rho^R\circ \xi\circ s(x)(f).$
\end{enumerate}
where $\ker s,\ker t$ are kernel of $s, t$ and $\xi:\L\to\P$ is a linear section of $s$:
$s\circ\xi(x)=\alpha_\L(x),$ for all $x\in\L.$
\end{definition}

Let  $\C=(i;s,t:\P\to\L)$ and $\C'=(i';s',t':\P'\to\L')$ be two cat$^1$-Hom-Leibniz-Rinehart algebras.
Then a homomorphism between them is a homomorphism of Hom-Leibniz-Rinehart algebras $\Upsilon:\P\rightarrow\P'$
satisfying the following additional conditions:
$$\Upsilon(\L)\subseteq\L',\quad s'\Upsilon=\Upsilon|_{\mathcal{L}}s, \quad t'\Upsilon=\Upsilon|_{\mathcal{L}}t.$$

The category of crossed modules of Hom-Leibniz-Rinehart algebras will be denoted by $\mathbf{Cat}$.

\begin{theorem}
Let $\partial:\M\to\L$ be a crossed module of Hom-Leibniz-Rinehart algebras over $(\A,\varphi)$. Then we can obtain a cat$^1$-Hom-Leibniz-Rinehart algebra $(i;s,t:\L\rtimes\M\to\L)$,
where the structural homomorphisms $s,t$ and the embedding $i$ are given by:
$$s(x,m)=\alpha_\L(x),\ \ t(x,m)=\alpha_\L(x)+\partial (\alpha_\M(m)),\ \  i(x)=(x,0).$$
The existence of  $\xi$ is ensured by $$\xi (x)=(x,Dx)$$ for a derivation $D:\L\to\M$ in the crossed module of Hom-Leibniz-Rinehart algebras.
\end{theorem}

\begin{proof} Firstly, we need to verify that $s$ and $t$ are Hom-Leibniz algebra homomorphisms. Since
\begin{eqnarray*}
s([(x,m),(x',m')]_{\L\rtimes\M})
&=&s([x,x']_\L,[m,m']_\M+[\alpha_\L(x),m']+[m,\alpha_\L(x')])\\
&=&\alpha_\L([x,x']_\L),
\end{eqnarray*}
\begin{eqnarray*}
[s(x,m),s(x',m')]_\L
&=&[\alpha_\L(x),\alpha_\L(x')]_\L.
\end{eqnarray*}
Then we get
$$s[(x,m),(x',m')]_{\L\rtimes\M}=[s(x,m),s(x',m')]_\L.$$
Since
$$s(\alpha(x,m))=s(\alpha_\L (x),\alpha_\M (m))=\alpha_\L(\alpha_\L (x)),$$
$$\alpha_\L(s(x,m))=\alpha_\L(\alpha_\L (x)),$$
we have $$s(\alpha(x,m))=\alpha_\L(s(x,m)).$$
Thus  we obtain that $s$ is a Hom-Leibniz algebra homomorphism.

For $t$, since
\begin{eqnarray*}
&&t([(x,m),(x',m')]_{\L\rtimes\M})\\
&=&t([x,x']_\L,[m,m']_\M+[\alpha_\L(x),m']+[m,\alpha_\L(x')])\\
&=&\alpha_\L([x,x']_\L)+\partial(\alpha_\M([m,m']_\M))\\
&&+\partial(\alpha_\M([\alpha_\L(x),m']))
+\partial(\alpha_\M([m,\alpha_\L(x')]))\\
&=&\alpha_\L([x,x']_\L)+\partial(\alpha_\M([m,m']_\M))\\
&&+\partial[\alpha_\L(x),\alpha_\M(m')]
+\partial[\alpha_\M(m),\alpha_\L(x')]\\
&=&\alpha_\L([x,x']_\L)+\partial(\alpha_\M([m,m']_\M))\\
&&+[\alpha_\L(x),\partial(\alpha_\M(m'))]_\L+[\partial(\alpha_\M(m)),\alpha_\L(x')]_\L,\ by\ (CM1)
\end{eqnarray*}
\begin{eqnarray*}
&&[t(x,m),t(x',m')]_\L\\
&=&[\alpha_\L(x)+\partial(\alpha_\M(m)),\alpha_\L(x')+\partial(\alpha_\M(m'))]_\L\\
&=&[\alpha_\L(x),\alpha_\L(x')]_\L+[\alpha_\L(x),\partial(\alpha_\M(m'))]_\L\\
&&+[\partial(\alpha_\M(m)),\alpha_\L(x')]_\L+[\partial(\alpha_\M(m)),\partial(\alpha_\M(m'))]_\L,
\end{eqnarray*}
Then we have $$t([(x,m),(x',m')]_{\L\rtimes\M})=[t(x,m),t(x',m')]_\L.$$
Since
$$t(\alpha(x,m))=t(\alpha_\L(x),\alpha_\M(m))=\alpha_\L(\alpha_\L(x))+\partial(\alpha_\M(\alpha_\M(m))),$$
\begin{eqnarray*}
\alpha_\L(t(x,m))
&=&\alpha_\L(\alpha_\L(x)+\partial(\alpha_\M(m)))\\
&=&\alpha_\L(\alpha_\L(x))+\alpha_\L(\partial(\alpha_\M(m)))\\
&=&\alpha_\L(\alpha_\L(x))+\partial(\alpha_\M(\alpha_\M(m))),
\end{eqnarray*}
Then we get $$t(\alpha(x,m))=\alpha_\L(t(x,m)).$$
Thus we obtain that $t$ is a Hom-Leibniz algebra homomorphism.

Secondly, we need to verify that the axioms in the Definition \ref{DEF22} are satisfied.

(Cat1):
\begin{eqnarray*}
s\circ i\circ t(x,m)
&=&s\circ (\alpha_\L(x)+\partial(\alpha_\M(m)))\\
&=&s(\alpha_\L(x)+\partial(\alpha_\M(m)),0)\\
&=&\alpha_\L(\alpha_\L(x)+\partial(\alpha_\M(m)))\\
&=&\alpha_\L(\alpha_\L(x))+\partial(\alpha_\M(\alpha_\M(m)))\\
&=&t(\alpha_\L(x),\alpha_\M(m))\\
&=&t\circ\alpha(x,m),
\end{eqnarray*}
\begin{eqnarray*}
t\circ i\circ s(x,m)
&=&t\circ i(\alpha_\L(x))\\
&=&t(\alpha_\L(x),0)\\
&=&\alpha_\L(\alpha_\L(x))+\partial(\alpha_\M(0))\\
&=&\alpha_\L(\alpha_\L(x))\\
&=&s(\alpha_\L(x),\alpha_\M(m))\\
&=&s\circ\alpha(x,m).
\end{eqnarray*}
Hence (Cat1) is satisfied.

(Cat2): By definition
$$\ker s=\{(0,\alpha_\M(m))|m\in\M\},\quad\ker t=\{(-\partial m',\alpha_\M(m'))|m\in\M\}$$
Assume $v\in \ker s, w\in \ker t,$
 then
\begin{eqnarray*}
&&[v,w]\\
&=&[(0,\alpha_\M(m)),(-\partial m',\alpha_\L(m'))]_{\L\rtimes\M}\\
&=&([0,-\partial m']_\L,[\alpha_\M(m),\alpha_\M(m')]_\M+[\alpha_\L(0),\alpha_\M(m')]+[\alpha_\M(m),\alpha_\L(-\partial m')])\\
&=&(0,[\alpha_\M(m),\alpha_\M(m')]_\M-[\alpha_\M(m),\alpha_\L(\partial m')])\\
&=&(0,[\alpha_\M(m),\alpha_\M(m')]_\M-[\alpha_\M(m),\partial\circ\alpha_\M(m')]) \\
&=&(0,[\alpha_\M(m),\alpha_\M(m')]_\M-[\alpha_\M(m),\alpha_\M(m')]_\M) \\
&=&0.
\end{eqnarray*}
That is $[\ker s,\ker t]=0$. Similarly, one can prove that $[\ker t,\ker s=0$. Thus (Cat2) is satisfied.

(Cat3):
\begin{eqnarray*}
t\circ\xi(fx)
&=&t\big(fx,D(fx)\big)\\
&=&\alpha_\L(fx)+\partial(\alpha_\M(Dfx))\\
&=&\varphi(f)\alpha_\L(x)+\partial(\varphi(f)\alpha_\M(Dx))\\
&=&\varphi(f)\alpha_\L(x)+\varphi(f)\partial(\alpha_\M(Dx))\\
&=&\varphi(f)\big(\alpha_\L(x)+\partial(\alpha_\M(Dx))\big)\\
&=&\varphi(f)t(x,Dx)\\
&=&\varphi(f)t\circ\xi(x).
\end{eqnarray*}

(Cat4): By section axiom and condition (CM4), the composition of the following maps is zero
$$\rho_\L\circ s\circ\xi(x)(f)=\rho_\L(s(x,Dx))(f)=\rho_\L(\partial(Dx))(f)=0.$$
Thus we get (Cat4).

Conversely, given a cat$^1$-Hom-Leibniz-Rinehart algebra, let $\M=\ker s$ and $\partial:\M\to \L$ be the composition $\partial=t|_{\ker s}=t\circ i: \ker s\to\L\rtimes\M\to \L$.
By condition (Cat1), we get $s|_{\L}=t|_{\L}=\alpha_{\L}$.

(CM0):
\begin{eqnarray*}
\alpha_{\L} \circ \partial ( m)=s(\partial m,0)=s\circ i(\partial m)=s\circ i\circ t(0,m),
\end{eqnarray*}
\begin{eqnarray*}
\partial\circ\alpha_{\M}(m)=t(0,\alpha_{\M}(m))=t\circ \alpha(0,m).
\end{eqnarray*}
Hence (CM0) is satisfied.

(CM1): For $\forall x\in\L, \exists x'\in\L$, such that $\alpha_{\L}(x')=x$.
Then we have,
\begin{eqnarray*}
\partial[x,m]-[x,\partial m]_\L
&=&\partial[\alpha_\L(x'),m]-[\alpha_\L(x'),\partial m]_\L\\
&=&t(0,[\alpha_\L(x'),m])-[\alpha_\L(x'),t(0,m)]_\L\\
&=&t[(x',0),(0,m)]_{\L\rtimes\M}-[t(x',0),t(0,m)]_\L\\
&=&0,
\end{eqnarray*}
\begin{eqnarray*}
\partial[m,x]-[m,\partial x]_\L
&=&\partial[m,\alpha_\L(x')]-[\partial m,\alpha_\L(x')]_\L\\
&=&t(0,[m,\alpha_\L(x')])-[t(0,m),\alpha_\L(x')]_\L\\
&=&t[(0,m),(x',0)]_{\L\rtimes\M}-[t(0,m),t(x',0)]_\L\\
&=&0.
\end{eqnarray*}
(CM2):
\begin{eqnarray*}
(0,[\partial m,n]-[m,n]_\M)
&=&(0,[t(0,m),n])-(0,[m,n]_\M)\\
&=&[(\alpha_\L^{-1}(t(0,m)),0),(0,n)]_{\L\rtimes\M}-[(0,m),(0,n)]_{\L\rtimes\M}\\
&=&[(\alpha_\L^{-1}(t(0,m),0)-(0,m),(0,n)]_{\L\rtimes\M}.
\end{eqnarray*}
It is clear that $(\alpha_\L^{-1}(t(0,m),0)-(0,m)\in $Ker $t$, since $t(\alpha_\L^{-1}(t(0,m),0)=t(0,m)$. Therefore, due to (Cat2), we have that
\begin{eqnarray*}
[(\alpha_\L^{-1}(t(0,m),0)-(0,m),(0,n)]_{\L\rtimes\M}=0,
\end{eqnarray*}
\begin{eqnarray*}
((0,[m,\partial n]-[m,n]_\M)
&=&(0,[m,t(0,n)])-(0,[m,n]_\M)\\
&=&[(0,m),(\alpha_\L^{-1}(t(0,n)),0)]_{\L\rtimes\M}-[(0,m),(0,n)]_{\L\rtimes\M}\\
&=&[(0,m),(\alpha_\L^{-1}(t(0,n))-(0,n)]_{\L\rtimes\M}).
\end{eqnarray*}

(CM3): For $\forall m\in\M, \exists x\in\L$, such that $D(x)=m$.
Then we have,
\begin{eqnarray*}
\partial (fm)=\partial(\varphi(g)(Dx))=\partial(D(gx))=t(0,D(gx))=t|_\M \circ \xi(gx),
\end{eqnarray*}
\begin{eqnarray*}
f(\partial m)=ft(0,Dx)=\varphi(g)t|_\M \circ (x).
\end{eqnarray*}
Hence (CM3) is satified.

(CM4): By contion (Cat4), we get
\begin{eqnarray*}
\rho^L_\L(\partial(m))(f)=\rho^L_\L\circ t|_{Kers}(Dx)(f)=\rho^L_\L\circ\xi\circ s(x)(f)=0.
\end{eqnarray*}
The proof is completed.
\end{proof}

\begin{theorem}
There is a category equivalence between the category of crossed modules of Hom-Leibniz-Rinehart algebras and the category
cat$^1$-Hom-Leibniz-Rinehart algebras.
\end{theorem}

\begin{proof} Given a homomorphism of crossed modules of Hom-Leibniz-Rinehart algebras $\Phi:\M\rightarrow\M'$ and $\Psi:\L\rightarrow\L'$,
the corresponding homomorphism of cat$^1$-Hom-Leibniz-Rinehart algebras is defined by
$$\Upsilon(x,m)=(\Phi(x),\Psi(m)).$$
Conversely,  given a homomorphism of cat$^1$-Hom-Leibniz-Rinehart algebras $\Upsilon:\L\rtimes\M\rightarrow\L'\rtimes\M'$
 the corresponding homomorphism of crossed modules of Hom-Leibniz-Rinehart algebras is give  by
 $\Phi:=\Upsilon|_{\M}$ and $\Psi:=\Upsilon|_{\L}$.
The proof that these two functors establish an equivalence between the two categories is similar as in \cite{Zhang1}.
The details is omitted.
\end{proof}

\section{Cohomology of Hom-Leibniz-Rinehart algebras}

\begin{definition}
Let $(\L, \alpha_\L,\rho^L_\L,\rho^R_\L)$ be a Hom-Leibniz-Rinehart algebra
and $(\M,\alpha_\M)$ be a representation of $\L.$
Consider the space of $(\A,\varphi)$-modules
$$
C^{*}(\L, \M):=\oplus_{n \geq 0} C^{n}(\L, \M)
$$
where $C^{n}(\L, \M) \subseteq \operatorname{Hom} (\otimes^{n}\L, \M )$ consisting of elements $\omega\in \operatorname{Hom} (\otimes^{n} \L, \M )$ satisfying conditions:
$$\omega (\alpha_\L (x_{1} ), \cdots, \alpha_\L (x_{n} ) )=\alpha_\M  (\omega (x_{1}, x_{2}, \cdots, x_{n} ) )$$
for all $x_{i} \in \L, 1 \leq i \leq n.$
 The coboundary map $\delta :  C^n  (\L, \M)  \rightarrow  C^{n+1}  (\L, \M)$ is given by
\begin{align*}
~&(\delta \omega) (x_1, \ldots, x_{n+1}) \\
=~& \sum_{i=1}^{n} (-1)^{i+1} [\alpha_{\L}^{n-1}(x_i), \omega (x_1, \ldots, \widehat{x_i}, \ldots, x_{n+1})] \\
~&+(-1)^{n+1} [\omega (x_1,\ldots, x_{n}),\alpha_{\L}^{n-1}(x_{n+1} )] \\
~&+ \sum_{1 \leq i <  j \leq n+1} (-1)^{i} \omega (\alpha_{\L}  (x_1), \ldots, \widehat{\alpha_{\L} ({x_i})}, \ldots, \alpha_{\L} (x_{j-1}), [x_i, x_j], \alpha_{\L} (x_{j+1}),\ldots,\alpha_{\L} ( x_{n+1})),
\end{align*}
for $\omega \in C^n  (\L, \M)$ and $x_1, \ldots, x_{n+1} \in \L.$
\end{definition}

\begin{proposition}
If $\omega \in C^{n}(\L, \M)$, then $\delta \omega\in C^{n+1}(\L, \M)$ and $\delta^{2}=0.$
\end{proposition}

\pf First, we need to check that
$$\delta \omega (\alpha_\L (x_{1} ), \alpha_\L (x_{2} ) \cdots, \alpha_\L (x_{n+1} ) )=\alpha_\M  (\delta \omega (x_{1}, x_{2}, \cdots, x_{n+1} ) )$$
 for all $x_{i} \in \L, 1 \leq i \leq n+1.$ We will use the fact that $\omega\circ \alpha_\L=\alpha_\M  \circ \omega$, $[\alpha_\L(x), \alpha_\M (m)]=\alpha_\M [x, m]$ and $[\alpha_\M (m), \alpha_\L(x)]=\alpha_\M [m, x]$ as $\omega\in C^{n}(\L, \M).$
\begin{eqnarray*}
&&\delta \omega (\alpha_\L (x_{1} ), \cdots, \alpha_\L (x_{n+1} ) )\\
&=& \sum_{i=1}^{n+1}(-1)^{i+1} [\alpha_\L^{n} (x_{i} ), \omega (\alpha_\L (x_{1} ), \cdots, \alpha_\L (\widehat{x}_{i} ), \cdots, \alpha_\L (x_{n+1} ) ) ] \\
&&+(-1)^{n+1} [\omega (\alpha_\L (x_{1} ), \cdots, \alpha_\L(x_{n})),\alpha_{\L}^{n}(x_{n+1})] \\
&&+ \sum_{1 \leq i <  j \leq n+1} (-1)^{i} \omega (\alpha^2_{\L}  (x_1), \ldots, \widehat{\alpha^2_{\L} ({x_i})}, \ldots, \alpha^2_{\L} (x_{j-1}),\\
 &&\qquad [x_i, x_j], \alpha^2_{\L} (x_{j+1}),\ldots,\alpha^2_{\L} ( x_{n+1}))\\
&=& \sum_{i=1}^{n}(-1)^{i+1} [\alpha_\L (\alpha_\L^{n-1} (x_{i} ) ), \alpha_\M  (\omega (x_{1}, \cdots, \widehat{x_{i}}, \cdots, x_{n+1} ) ) ]\\
&&+(-1)^{n+1} [\alpha_\M\omega (x_{1}, \cdots, x_{n}),\alpha_{\L} (\alpha_{\L}^{n-1}(x_{n+1}) )] \\
&&+ \sum_{1 \leq i <  j \leq n+1} (-1)^{i}\alpha_\M  (\omega (\alpha_{\L}  (x_1), \ldots, \widehat{\alpha_{\L} ({x_i})}, \ldots, \alpha_{\L} (x_{j-1}), \\
&&\qquad [x_i, x_j], \alpha_{\L} (x_{j+1}),\ldots,\alpha_{\L} ( x_{n+1}))) \\
&=& \alpha_\M  (\delta \omega (x_{1}, x_{2}, \cdots, x_{n+1} ) ).
\end{eqnarray*}
Further, $\delta^{2}=0$ follows from the direct but a long calculation.
\qed

By the above proposition, we obtain a cochain complex. The resulting cohomology of the
cochain complex is called the cohomology space of Hom-Leibniz-Rinehart algebra, and we denote this cohomology as $H^{}(\L,\M).$

We will use this cohomology when we consider extensions of a Hom-Leibniz-Rinehart algebras.
The second cohomology group $H^2(\L, \M)$ can be described as the equivalence classes of abelian extensions of Hom-Leibniz-Rinehart algebra.
%
\begin{definition}
A short exact sequence
\[
\xymatrix@C=0.5cm{
  0 \ar[r] & \M \ar[r]^{i} & \L' \ar[r]^{\pi} & \L \ar[r] & 0 }
\]
is called an abelian extension of the Hom-Leibniz-Rinehart algebra $\L$ by $\M$ if
\begin{enumerate}
\item[(E1)]$[i(m), x]=i([m,\pi(x)]),$
\item[(E2)]$[x,i(m)]=i([\pi(x), m]),$
\item[(E3)]$[i(m),i(n)]=0=[i(n),i(m)],$
\end{enumerate}
for all $m,n \in \M, x \in \L'.$
\end{definition}
Next, we will show that the second cohomology space $H^{2}(\L, \M)$ of a Hom-Leibniz-Rinehart algebra $(\mathcal{L}, \alpha_{\L})$ with coefficients in $(\M, \alpha_\M )$ classifies $\A$-split abelian extensions of $(\mathcal{L}, \alpha_{\L})$ by $(\M, \alpha_\M ).$

This result generalises the well-known classification theorems for the classical cases of a Lie algebras [8] and Lie-Rinehart algebras [9].

\begin{lemma}
Let $h$ be a representative of the cohomology class $[h] \in H^{2}(\L, \M).$ Consider a Hom-Leibniz-Rinehart algebra $ (\mathcal{L}^{\prime}, \alpha_{\L'} ):= (\L^{\prime},[-,-]^{\prime}, \varphi, \alpha_{\L'}, \rho^{L}_{\L'},\rho^{R}_{\L'} )$, where the structure constraints are given as follows:

(1) $\L^{\prime}=\L \oplus \M$, a direct sum of $\A$-modules;

(2) $[(x, m),(y, n)]^{\prime}=([x, y],[x, n]+[m,y]+h(x, y))$;

(3) $\alpha_{\L'}((x, m))=(\alpha_{\L}(x), \alpha_\M (m))$;

(4) $\rho^{L}_{\L'}(x, m)=\rho^{L}_{\L}(x)=\rho^{L}_{\L}(\pi(x, m))$,
$\rho^{R}_{\L'}(x, m)=\rho^{R}_{\L}(x)=\rho^{R}_{\L}(\pi(x, m))$\\
for all $x, y \in \L, m, n \in \M$ and $\pi: \L^{\prime}  \rightarrow \L$ defined as $\pi(x, m)=x.$ Then
\[
\xymatrix@C=0.5cm{
  0 \ar[r] & (\mathcal{M}, \alpha_\M ) \ar[r]^{i} & (\L', \alpha_{\L'}) \ar[r]^{\pi} & (\mathcal{L}, \alpha_{\L}) \ar[r] & 0 }
\]
is an $\A$-split abelian extension of $(\mathcal{L}, \alpha_{\L})$ by $(\M, \alpha_\M )$, where $i: \M  \rightarrow \L^{\prime}$ is defined by $i(m)=(0, m).$
\end{lemma}
\pf
Firstly, we need to verify that $(\L', \alpha_{\L'})$ is a Hom-Leibniz-Rinehart algebra. By direct calculations, we will only verify the conditions (H01), (H11), (H21), (H31) and (H41) hold.
(H01):
$$
\begin{aligned}
\alpha_{\L'}(f(x,m))
=&\alpha_{\L'}(fx,fm)
=(\alpha_{\L}(fx),\alpha_{\M}(fm))\\
=&(\varphi(f)\alpha_{\L}(x),\varphi(f)\alpha_{\M}(x))
=\varphi(f)\alpha_{\L'}(x,m).
\end{aligned}
$$
(H11):
$$
\begin{aligned}
\rho^{L}_{\L'}(\alpha_{\L'}(x,m))\circ\varphi(f)
=&\rho^{L}_{\L'}(\alpha_{\L}(x),\alpha_{\M}(m))\circ\varphi(f)
=\rho^{L}_{\L}(\alpha_{\L}(x))\circ\varphi(f)\\
=&\varphi\circ\rho^{L}_{\L}(x)(f)
=\varphi\circ\rho^{L}_{\L'}(x,m)(f).
\end{aligned}
$$
(H21):
\begin{eqnarray*}
&&\rho^{R}_{\L'}(\alpha_{\L'}(y,n))\circ\rho^{R}_{\L'}(x,m)(f)\\
&=&\rho^{R}_{\L'}(\alpha_{\L}(y),\alpha_{\M}(n))\circ\rho^{R}_{\L'}(x,m)(f)\\
&=&\rho^{R}_{\L}(\alpha_{\L}(y))\circ\rho^{R}_{\L}(x)(f)\\
&=&-\rho^{R}_{\L}(\alpha_{\L}(y))\circ\rho^{L}_{\L}(x)(f)\\
&=&-\rho^{R}_{\L'}(\alpha_{\L}(y),\alpha_{\M}(n))\circ\rho^{L}_{\L'}(x,m)(f)\\
&=&-\rho^{R}_{\L'}(\alpha_{\L'}(y,n))\circ\rho^{L}_{\L'}(x,m)(f).
\end{eqnarray*}
(H22):
\begin{eqnarray*}
&&\rho^{L}_{\L'}([(x,m),(y,n)]')\circ\varphi(f)\\
&=&\rho^{L}_{\L'}([x,y],[x,n]+[m,y]+h(x,y))\circ\varphi(f)\\
&=&\rho^{L}_{\L}([x,y])\circ\varphi(f)\\
&=&\rho^{L}_\L([x,y])\circ\varphi(f)\\
&=&\rho^{L}_\L(\alpha_\L(x))\circ\rho^{L}_\L(y)(f)-\rho^{L}_\L(\alpha_\L(y))\circ\rho^{L}_\L(x)(f)\\
&=&\rho^{L}_{\L'}(\alpha_\L(x),\alpha_\M(m))\circ\rho^{L}_{\L'}(y,n)(f)-\rho^{L}_{\L'}(\alpha_\L(y),\alpha_\M(n))\circ\rho^{L}_{\L'}(x,m)(f)\\
&=&\rho^{L}_{\L'}(\alpha_{\L'}(x,m))\circ\rho^{L}_{\L'}(y,n)(f)-\rho^{L}_{\L'}(\alpha_{\L'}(y,n))\circ\rho^{L}_{\L'}(x,m)(f).
\end{eqnarray*}
(H31):
\begin{eqnarray*}
&&[(x,m),f(y,n)]\\
&=&[(x,m),(fy,fn)]\\
&=&([x,fy],[x,fn]+[m,fy]+h(x,fy))\\
&=&(\varphi(f)[x, y]+\rho^{L}_\L(x)(f)\alpha_\L(y),\\
&&+\varphi(f)[x,n]+\rho^{L}_\L(x)(f)\alpha_\M(n)+\varphi(f)[m,y]+\varphi(f)h(x,y))\\
&=&(\varphi(f)[x, y],\varphi(f)[x, n]+\varphi(f)[m,y]+\varphi(f)h(x,y))\\
&&+(\rho^{L}_\L(x)(f)\alpha_\L(y),\rho^{L}_\L(x)(f)\alpha_\M(n))\\
&=&\varphi(f)([x, y],[x, n]+[m,y]+h(x,y))+\rho^{L}_\L(x)(f)(\alpha_\L(y),\alpha_\M(n))\\
&=&\varphi(f)([x, y],[x, n]+[m,y]+h(x,y))+\rho^{L}_{\L'}(x,m)(f)(\alpha_\L(y),\alpha_\M(n))\\
&=&\varphi(f)([x,m],[y,n])+\rho^{L}_{\L'}(x,m)(f)(\alpha_{\L'}(y,n)).
\end{eqnarray*}
(H41):
$$
\begin{aligned}
\rho^{L}_{\L'}(f(x,m))
=\rho^{L}_{\L'}(fx,fm)
=\rho^{L}_{\L}(fx)
=\varphi(f)\rho^{L}_{\L}(x)
=\varphi(f)\rho^{L}_{\L'}(x,m).
\end{aligned}
$$
Secondly, we need to verify that the axioms in the Definition 3.3 are satisfied.\\
(E1):
$$
\begin{aligned}\relax [i(m),i(n)]
=[(0,m),(0,n)]
=([0,0],[0,n]+[m,0]+f(0,0))
=0.
\end{aligned}
$$
(E2):
$$
\begin{aligned}\relax[i(m),(x,m)]
=&[(0,m),(x,m)]
=([0,x],[0,m]+[m,x]+f(0,x))\\
=&(0,[m,x])
=i([m,x])
=i([m,\pi(x,m)]).
\end{aligned}
$$
(E3):
$$
\begin{aligned}\relax [(x,m),i(m)]
=&[(x,m),(0,m)]
=([x,0],[m,0]+[x,m]+f(x,0))\\
=&(0,[x,m])
=i([x,m])
=i([\pi(x,m),m]).
\end{aligned}
$$
The proof is completed.
\qed

\begin{lemma}
For a cohomology class in $H^{2}(\L, \M)$ there is a unique equivalence class of $\A$-split abelian extensions of $(\mathcal{L}, \alpha_{\L})$ by $(\M, \alpha_\M ).$
\end{lemma}
\pf
Suppose we take an another representative $h^{\prime}$ of the cohomology class $[h] \in H^{2}(\L, \M)$ and get an extension $ (\mathcal{L}^{\prime \prime}, \alpha_{\L''} )$ as above. Since $h$ and $h^{\prime}$ represent the same cohomology class $[h]$, we have $h-h^{\prime}=\delta g$ for some $g \in C^{1}(\L, \M).$ Then the map $F: (\mathcal{L}^{\prime}, \alpha_{\L'} )  \rightarrow (\mathcal{L}^{\prime \prime}, \alpha_{\L''} )$ defined by $F(x, m)=(x, m+g(x))$ gives an isomorphism of the above extensions obtained by using $h$ and $h^{\prime}$ respectively. Thus for a cohomology class in $H^{2}(\L, \M)$ there is a unique equivalence class of $\A$-split abelian extensions of $(\mathcal{L}, \alpha_{\L})$ by $(\M, \alpha_\M ).$
\qed
\begin{theorem}
There is a one-to-one correspondence between the equivalence classes of $\A$-split abelian extensions of a Hom-Leibniz-Rinehart algebra
$(\mathcal{L}, \alpha_{\L})$ by $(\M, \alpha_\M )$ and the cohomology classes in $H^{2}(\L, \M).$
\end{theorem}

\pf
Let
\[
\xymatrix@C=0.5cm{
  0 \ar[r] & (\mathcal{M}, \alpha_\M ) \ar[r]^{i} & (\L', \alpha_{\L'}) \ar[r]^{\sigma} & (\mathcal{L}, \alpha_{\L}) \ar[r] & 0 }
\]
be an $\A$-split abelian extension of the Hom-Leibniz-Rinehart algebra $(\mathcal{L}, \alpha_{\L})$ by $(\M, \alpha_\M ).$ We will first show that we can define a 2-cocycle in $C^{2}(\L, \M)$ which is independent of a section for the map $\sigma.$

First, we fix a section $\tau: \L  \rightarrow \L'$ for the map $\sigma.$ Now consider the map $G: \L \oplus \M  \rightarrow \L'$ given by
$$
G(x, m)=\tau(x)+i(m) .
$$
Then it follows that $G$ is an injective $\A$-module homomorphism. In fact, $G$ is an isomorphism of $\A$-modules.

Define a 2-cochain $\Omega_{\tau} \in C^{2}(\L, \M)$ by the following expression;
$$
\Omega_{\tau}(x, y)=i^{-1}([\tau(x), \tau(y)]-\tau([x, y])) \text {, }
$$
for all $x, y \in \L.$ Here we have\\
(1) $\Omega_{\tau} \circ \alpha_{\L}=\alpha_\M  \circ \Omega_{\tau}$, which follows by the relations $\tau \circ \alpha_{\L}=\alpha_{\L'} \circ \tau$, and $\alpha_{\L'} \circ i=i \circ \alpha_\M $;\\
\begin{eqnarray*}
&&\Omega_{\tau}(\alpha_{\L}(x),\alpha_{\L}(y))\\
&=&i^{-1}([\tau(\alpha_{\L}(x)),\tau(\alpha_{\L}(y))]-\tau[\alpha_{\L}(x),\alpha_{\L}(y)])\\
&=&i^{-1}([\alpha_{\L'}(\tau(x)),\alpha_{\L'}(\tau(y))]-\alpha_{\L'}(\tau[x,y]))\\
&=&\alpha_{\M}(i^{-1}([\tau(x),\tau(y)]-\tau[x,y]))\\
&=&\alpha_{\M}(\Omega_{\tau}(x,y)).
\end{eqnarray*}
The claim is proved.\\
(2) $\delta (\Omega_{\tau} )=0$, which follows using Hom-Leibniz identity for $ (\L^{\prime}, \alpha_{\L'} ).$
\begin{eqnarray*}
&&\delta (\Omega_{\tau} )(x,y,z)\\
&=&[\alpha_{\L}^{2}(x),i^{-1}([\tau(y),\tau(z)]-\tau[y,z])]-[\alpha_{\L}^{2}(y),i^{-1}([\tau(x),\tau(z)]-\tau[x,z])]\\
&&-[i^{-1}([\tau(x),\tau(y)]-\tau[x,y]),\alpha_{\L}^{2}(z)]-i^{-1}([\tau[x,y],\tau(\alpha_{\L}(z))]-\tau[[x,y],\alpha_{\L}(z)])\\
&&-i^{-1}([\tau(\alpha_{\L}(y)),\tau[x,z]]-\tau[\alpha_{\L}(y),[x,z]])+i^{-1}([\tau(\alpha_{\L}(x)),\tau[y,z]]-\tau[\alpha_{\L}(x),[y,z]])\\
&=&[\alpha_{\L}^{2}(x),[i^{-1}\tau(y),i^{-1}\tau(z)]]-[\alpha_{\L}^{2}(x),i^{-1}\tau[y,z]]\\
&&-[\alpha_{\L}^{2}(y),[i^{-1}\tau(x),i^{-1}\tau(z)]]+[\alpha_{\L}^{2}(y),i^{-1}\tau[x,z]]\\
&&-[[i^{-1}\tau(x),i^{-1}\tau(y)],\alpha_{\L}^{2}(z)]+[i^{-1}\tau[x,y],\alpha_{\L}^{2}(z)]\\
&&-i^{-1}([\tau[x,y],\tau(\alpha_{\L}(z))]-\tau[[x,y],\alpha_{\L}(z)])\\
&&-i^{-1}([\tau(\alpha_{\L}(y)),\tau[x,z]]-\tau[\alpha_{\L}(y),[x,z]])\\
&&+i^{-1}([\tau(\alpha_{\L}(x)),\tau[y,z]]-\tau[\alpha_{\L}(x),[y,z]]).
\end{eqnarray*}
By using Hom-Leibniz identity, the claim is proved. Consequently, we get that $\Omega_{\tau}$ is a 2-cocycle in $C^{2}(\L, \M).$

Next we show that the class $[\Omega_{\tau} ] \in H^{2}(\L, \M)$ does not depend on the Hom-section $\tau.$ Suppose $\tau^{\prime}: \L \rightarrow \L^{\prime}$ is another Hom-section of $\sigma$ and let $\Omega_{\tau'}$ be the corresponding $2$-cocycle defined by using $\tau'$ instead of $\tau.$ Then the resulting 2-cocycle $\Omega_{\tau^{\prime}}$ is cohomologous to $\Omega_{\tau}$. This follows from the fact that $\Omega_{\tau^{\prime}}-\Omega_{\tau}=\delta (i^{-1} \circ (\tau^{\prime}-\tau ) )$ for $ (i^{-1} \circ (\tau^{\prime}-  \tau)) \in C^{1}(\L, \M)$. Thus for a given $\A$-split abelian extension of $(\mathcal{L}, \alpha_{\L})$ by $(\M, \alpha_\M )$, there exists a unique cohomology class $ [\Omega_{\tau} ] \in H^{2}(\L, \M)$.

One check that for two equivalent $\A$-split abelian extensions, the associated 2-cocycles are cohomologous.

Let
\[
\xymatrix@C=0.5cm{
  0 \ar[r] & (\mathcal{M}, \alpha_\M ) \ar[r]^{i'} & (\L'', \alpha_{\L''}) \ar[r]^{\sigma'} & (\mathcal{L}, \alpha_{\L}) \ar[r] & 0 }
\]
be another $\A$-split abelian extension of $(\mathcal{L}, \alpha_{\L})$ by $(\M, \alpha_{\M})$, and it is isomorphic to the extension:
\[
\xymatrix@C=0.5cm{
  0 \ar[r] & (\mathcal{M}, \alpha_\M ) \ar[r]^{i} & (\L', \alpha_{\L'}) \ar[r]^{\sigma} & (\mathcal{L}, \alpha_{\L}) \ar[r] & 0 }
\]
Suppose the map $\Phi: (\mathcal{L}^{\prime}, \alpha_{\L'}) \rightarrow (\mathcal{L}^{\prime \prime}, \alpha_{\L''})$ is an isomorphism of these extensions, that is the following diagram commutes:
\[\xymatrix{0 \ar[r] &
(\mathcal{M}, \alpha_\M ) \ar[r]^{i} \ar@{=}[d] &
(\L', \alpha_{\L'}) \ar[r]^{\sigma} \ar[d]^{\Phi}&
(\mathcal{L}, \alpha_{\L}) \ar[r] \ar@{=}[d]& 0 \\
0 \ar[r] & (\mathcal{M}, \alpha_\M ) \ar[r]^{i'} & (\L'', \alpha_{\L''}) \ar[r]^{\sigma'} & (\mathcal{L}, \alpha_{\L}) \ar[r] & 0
  }
\]

Now we will show that for a Hom-section $\tau:(\mathcal{L}, \alpha_{\L}) \rightarrow (\mathcal{L}^{\prime}, \alpha_{\L'})$ of $\sigma$ and $\tau^{\prime}:(\mathcal{L}, \alpha_{\L}) \rightarrow (\mathcal{L}^{\prime \prime}, \alpha_{\L''})$ of $\sigma^{\prime}$, the respective associated cocycles $\Omega_{\tau}$ and $\Omega_{\tau^{\prime}}$ are cohomologous. Consider, $\tau^{\prime \prime}=\Phi \circ \tau:(\mathcal{L}, \alpha_{\L}) \rightarrow$ $ (\mathcal{L}^{\prime \prime}, \alpha_{\L''})$ a Hom-section for $\sigma^{\prime}$. Then we have $\Omega_{\tau^{\prime \prime}}=\Omega_{\tau} .$ Here, $\Omega_{\tau^{\prime \prime}}$ and $\Omega_{\tau^{\prime}}$ are cohomologous in $H^{2}(\L, \M)$ because of the fact that $\tau^{\prime}$ and $\tau^{\prime \prime}$ both are Hom-sections of $\sigma^{\prime}$. Therefore, $\Omega_{\tau}$ and $\Omega_{\tau^{\prime}}$ are cohomologous in $H^{2}(\L, \M)$.

With Lemma $3.4$ and Lemma $3.5$ this proof is completed.
\qed

\section{$\alpha$-crossed module extensions of Hom-Leibniz-Rinehart algebras}
In this section, we introduce the concept of $\alpha$-crossed module of Hom-Leibniz-Rinehart algebras.
The goal of this section is to find the relationship between the third cohomology group of Hom-Leibniz-Rinehart algebras and the $\alpha$-crossed module extensions of Hom-Leibniz-Rinehart algebras.

\begin{definition} 
An $\alpha$-crossed module $\partial:(\V,\alpha_{\V})\rightarrow(\L,\alpha_{\L})$ of Hom-Leibniz-Rinehart algebras over $(\A,\varphi)$ consists of a Hom-Leibniz-Rinehart algebra
$(\L,\alpha_{\L})$, a Hom-Leibniz-$\A$-algebra $(\V,\alpha_\V)$ together with the Hom-action of $\L$ on $\V$ such that the following identities hold:
\begin{enumerate}
\item[($\alpha$-CM0)] $\alpha_\L\circ\partial(v)=\partial\circ\alpha_\V(v),$
\item[($\alpha$-CM1)] $\partial[\alpha_\L (x),v]=[x,\partial v]_\L,\quad \partial[v,\alpha_\L (x)]=[\partial v, x]_\L,$
\item[($\alpha$-CM2)] $[\partial(\alpha_\V (v)), v']=[v,v']_\V=[v,\partial(\alpha_\V(v'))],$
\item[($\alpha$-CM3)] $\partial(fv)=f\partial(v),$
\item[($\alpha$-CM4)] $\rho_\L^{L}(\partial(v))(f)=0=\rho_\L^{R}(\partial(v))(f)$,
\end{enumerate}
\noindent where $v, v'\in \V,x\in\L,f\in\A.$
\end{definition}

 Obviously, an $\alpha$-crossed module is a crossed module  in the sense of \cite{Cas20} when $\alpha_{\L}=\text{id}_{\L}$ and $\alpha_{\V}=\text{id}_{\V}$.
Now for a Hom-Leibniz-Rinehart algebra $(\L,\alpha_{\L})$ and a Hom-Leibniz $\A$-algebra $(\V,\alpha_{\V})$ where $\L$ acts on it. We consider direct sum $\L\oplus\V$ as an $\A$-module: $f(x,v)=(fx,fv).$
Define the anchors $\rho^{L},\rho^{R}:\L\oplus\V\to \Der_\varphi(\A)$:
$$\rho^{L}(x,v)=\rho^{L}_\L(x),\quad\rho^{R}(x,v)=\rho^{R}_\L(x),$$
the map $\alpha:\L\oplus\V\to\L\oplus\V$:
$$\alpha(x,v)=(\alpha_\L(x),\alpha_\V(v)),$$
and the bracket:
$$[(x,v),(x',v')]=([x,x']_\L,[v,v']_\V+[\alpha_\L(x),v']+[v, \alpha_\L(x')]).$$
for all $v,v' \in \V, x,x' \in \L $.

\begin{lemma}
The vector space $\L\oplus\V$ equipped with the above maps and bracket is a  Hom-Leibniz-Rinehart algebra if and only if the conditions (S11)-(S32) hold.
This is called an $\alpha$-semi-direct product of Hom-Leibniz-Rinehart algebras of $\L$ and $\V$, denoted by $\L\rtimes_\alpha\V$.
\end{lemma}

\begin{definition} 
Let $(\V, \L, \partial)$ and $(\V', \L', \partial')$ be two $\alpha$-crossed modules. A homomorphism of $\alpha$-crossed modules is a pair of Hom-Leibniz-Rinehart algebra homomorphisms $(\lambda, \mu)$  such that $\lambda:
(\V, \alpha_{\V}) \rightarrow(\V', \alpha_{\V'})$ and $\mu:(\L, \alpha_{\L}) \rightarrow(\L', \alpha_{\L'})$ satisfying $\lambda [x,v] = [\mu(x), \lambda (v)] $  for all $x \in \L,~ v \in \V$ and the following diagram commute
\[
\xymatrix{
\V \ar[r]^{\partial} \ar[d]_{\lambda} & \L \ar[d]^{\mu} \\
\V' \ar[r]_{\partial'} & \L'.
}
\]

\end{definition}

Given an $\alpha$-crossed module $(\V , \L, \partial)$, we consider $\mathfrak{g} = \text{coker}(\partial)$ and $\M = \text{ker} (\partial)$. The Hom-Leibniz-Rinehart algebra structure on $\L$ induces a Hom-Leibniz-Rinehart algebra structure on $\mathfrak{g}$ by $[\pi(x), \pi(y)] = \pi [x,y]$, where $\pi : \L \rightarrow \mathfrak{g}$ is the projection map. Moreover, the Hom-action of $\L$ on $\V$ induces a Hom-action of $\mathfrak{g}$ on $\M$ via $[x, i(m)] = [\pi (x),m]$, for $x \in \L,m \in \M$. Hence an $\alpha$-crossed module yields an exact sequence
\begin{align*}
0 \rightarrow \M \xrightarrow{i} \V \xrightarrow{\partial} \L \xrightarrow{\pi} \mathfrak{g} \rightarrow 0.
\end{align*}
We call $(\V , \L, \partial )$ an $\alpha$-crossed module over the Hom-Leibniz-Rinehart algebra $\mathfrak{g}$ with kernel a Hom-$\mathfrak{g}$-module $\M$. In this case, we call such an exact sequence an $\alpha$-crossed module extension of $\g$ by $\M$.

Two $\alpha$-crossed modules
$(\V, \L, \partial)$ and $(\V', \L', \partial')$ over $\mathfrak{g}$ with kernel $\M$ are said to be equivalent if there is a morphism of $\alpha$-crossed modules $(\V, \L, \partial) \rightarrow (\V', \L', \partial')$ which induces identity maps on $\mathfrak{g}$ and $\M$. Denote $\text{Cross}(\mathfrak{g}, \M)$ the equivalence classes of $\alpha$-crossed module extensions.
In the next theorem we see the relation between classes of $\alpha$-crossed modules and the third cohomology.

\begin{theorem}\label{last-thm}
There is a canonical map
\begin{align*}
\psi : \text{Cross}(\mathfrak{g}, \M) \rightarrow H^3  (\mathfrak{g}, \M).
\end{align*}
\end{theorem}

\pf Let

\begin{align*}
\E:0 \rightarrow \M \xrightarrow{i} \V \xrightarrow{\partial} \L \xrightarrow{\pi} \g\rightarrow 0
\end{align*}
be an $\alpha$-crossed module extension. Choose Hom-$\A$-linear sections $s : \mathfrak{g} \rightarrow \L$ with $\pi s = \text{id}_{\g}$ and $q : \text{Im}(\partial) \rightarrow \V$ with $\partial q = \text{id}_{\L}$. Since $\pi$ is a Hom-Leibniz-Rinehart algebra homomorphism, for all $x , y \in \mathfrak{g},$ we have
\begin{align*}
\pi ([s(x), s(y)] - s[x,y])&=[\pi s(x), \pi  s(y)] - \pi s[x,y]\\
&=[x,y]-[x,y]= 0.
\end{align*}
This shows that $[s(x), s(y)] - s[x,y]\in\ker(\pi) = \text{Im} (\partial)$.
Take
$$g (x,y) = q ([s(x), s(y)] - s[x,y]) \in \V.$$
Define a map $\theta_{\mathcal{E}, s, q} : \mathfrak{g}^{\otimes 3} \rightarrow \M$ by
\begin{align*}
\theta_{\mathcal{E}, s, q} (x,y,z) =~ &[\alpha_{\L}^{2}s(x), g(y,z)] - [\alpha_{\L}^{2}s(y), g(x,z)] - [g (x,y),\alpha_{\L}^{2}s(z)]\\
~& + g (\alpha_{\g}(x),[y, z]) - g (\alpha_{\g}(y),[x,z] ) - g ([x,y],\alpha_{\g}(z)).
\end{align*}
Since $\partial$ is a map of Hom-$\L$-modules, it follows that
\begin{eqnarray*}
&&\partial\theta_{\mathcal{E}, s, q} (x,y,z)\\
 &= &\partial[\alpha_{\L}^{2}s(x), g(y,z)] - \partial[\alpha_{\L}^{2}s(y), g(x,z)] - \partial[g (x,y),\alpha_{\L}^{2}s(z)]\\
&& + \partial g (\alpha_{\g}(x),[y, z]) -\partial g (\alpha_{\g}(y),[x,z] ) -\partial g ([x,y],\alpha_{\g}(z)]).\\
&= &\partial\big[\alpha_{\L}^{2}s(x), q([s(y),s(z)]-s[y,z])\big
] + \partial q([ \alpha_{\L}s(x),s[y, z]]-s[\alpha_{\g}(x),[y, z]])\\
& &-\partial[\alpha_{\L}^{2}s(y), q([s(x),s(z)]-s[x,z])] - \partial q([ \alpha_{\L}s(y),s[x, z]]-s[\alpha_{\g}(y),[x, z]])\\
&&-\partial[ q([s(x),s(y)]-s[x,z]),\alpha_{\L}^{2}s(z)] - \partial q([ s[x, y],\alpha_{\L}s(z)]-s[[x, y],\alpha_{\g}(z)])\\
&=& [\alpha_{\L}s(x),  ([s(y),s(z)]-s[y,z])] +    ([ \alpha_{\L}s(x),s[y, z]]-s[\alpha_{\g}(x),[y, z]])\\
&&- [\alpha_{\L}s(y),  ([s(x),s(z)]-s[x,z])] -    ([ \alpha_{\L}s(y),s[x, z]]-s[\alpha_{\g}(y),[x, z]])\\
&&- [  ([s(x),s(y)]-s[x,z]),\alpha_{\L}s(z)] -   ([ s[x, y],\alpha_{\L}s(z)]-s[[x, y],\alpha_{\g}(z)]) \ \ by\ \  (\alpha\text-CM1)\\
&=& [\alpha_{\L}s(x),  [s(y),s(z)]]-[\alpha_{\L}s(x),s[y,z]] +   [  \alpha_{\L}s(x),s[y, z]]-s[\alpha_{\g}(x),[y, z]]\\
&&- [\alpha_{\L}s(y),  [s(x),s(z)]]+[\alpha_{\L}s(y),s[x,z])] -    [ \alpha_{\L}s(y),s[x, z]]+s[\alpha_{\g}(y),[x, z]]\\
&&- [  [s(x),s(y)],\alpha_{\L}s(z)]+[s[x,y],\alpha_{\L}s(z)] -   [ s[x, y],\alpha_{\L}s(z)]+s[[x, y],\alpha_{\g}(z)]\\
&=&[\alpha_{\L}s(x),  [s(y),s(z)]]- [\alpha_{\L}s(y),  [s(x),s(z)]]- [  [s(x),s(y)],\alpha_{\L}s(z)]\\
&&-s[\alpha_{\g}(x),[y, z]]+s[\alpha_{\g}(y),[x, z]]+s[[x, y],\alpha_{\g}(z)].
\end{eqnarray*}
By Hom-Leibniz identity, we obtain $\partial (\theta_{\mathcal{E}, s, q} (x,y,z)) = 0.$

Therefore, $\theta_{\mathcal{E}, s, q} (x,y,z) \in \ker (\partial) = \M$. Hence $\theta_{\mathcal{E}, s, q} : \wedge^3 \mathfrak{g} \rightarrow \M$ is well defined.

For condition
$\theta_{\mathcal{E}, s, q}
(\alpha_{\mathfrak{g}}(x),\alpha_{\mathfrak{g}}(y),\alpha_{\mathfrak{g}}(z))=\alpha_{\M}(\theta_{\mathcal{E}, s, q}(x,y,z))$, we first check that
\begin{eqnarray*}
 g(\alpha_{\mathfrak{g}}x,\alpha_{\mathfrak{g}}y)
&=&q ([s(\alpha_\mathfrak{g}x), s(\alpha_\mathfrak{g}y)] - s[\alpha_\mathfrak{g}x,\alpha_\mathfrak{g}y])\\
&=&q ([\alpha_\L s(x), \alpha_\L s(y)] - \alpha_\L s[x,y])\\
&=&\alpha_\V (q([s(x),s(y)]-s[x,y]))\\
&=&\alpha_\V( g(x,y)).
\end{eqnarray*}
Then we can get
\begin{eqnarray*}
&&\theta_{\mathcal{E}, s, q}
(\alpha_{\mathfrak{g}}x,\alpha_{\mathfrak{g}}y,\alpha_{\mathfrak{g}}z)\\
&=&[s(\alpha^{3}_\mathfrak{g}x), g(\alpha_\mathfrak{g}y,\alpha_\mathfrak{g}z)] - [s(\alpha^{3}_\mathfrak{g}y), g(\alpha_\mathfrak{g}x,\alpha_\mathfrak{g}z)] - [g (\alpha_\mathfrak{g}x,\alpha_\mathfrak{g}y),s(\alpha^{3}_\mathfrak{g}z)]\\
&& + g (\alpha^{2}_\mathfrak{g}(x),[\alpha_\mathfrak{g}y, \alpha_\mathfrak{g}z]) - g (\alpha^{2}_\mathfrak{g}(y),[\alpha_\mathfrak{g}x,\alpha_\mathfrak{g}z] ) - g ([\alpha_\mathfrak{g}x,\alpha_\mathfrak{g}y],\alpha^{2}_\mathfrak{g}(z)])\\
&=&[\alpha^{3}_\L s(x), \alpha_{\V}g(y,z)] - [\alpha^{3}_ \L s(y),\alpha_{\V} g( x,z)] - [\alpha_{\V}g (x,y),\alpha^{3}_ \L s( z)]\\
&& + \alpha_{\V}g (\alpha_\mathfrak{g} (x),[y,  z]) -\alpha_{\V} g (\alpha_\mathfrak{g} (y),[ x,z] ) - \alpha_{\V}g ([ x,y],\alpha_ \mathfrak{g}(z)])\\
&=&\alpha_{\V}[\alpha_\L^{2} s(x), g(y,z)] - \alpha_{\V}[\alpha_ \L^{2} s(y), g( x,z)] - \alpha_{\V}[g (x,y),\alpha_ \L^{2} s( z)]\\
&& + \alpha_{\V}g (\alpha_\mathfrak{g} (x),[y,  z]) -\alpha_{\V} g (\alpha_\mathfrak{g} (y),[ x,z] ) - \alpha_{\V}g ([ x,y],\alpha_ \mathfrak{g}(z)])\\
&=&\alpha_{\V}(\theta_{\mathcal{E}, s, q}(x,y,z))\\
&=&\alpha_{\M}(\theta_{\mathcal{E}, s, q}(x,y,z)).
\end{eqnarray*}
Moreover, the map $\theta_{\mathcal{E}, s, q}$ is a $3$-cocycle ($\delta(\theta_{\mathcal{E}, s, q})=0$) in $H^3 (\mathfrak{g}, \M)$.

We next show that the class of $\theta_{\mathcal{E}, s, q}$ in $H^3 (\mathfrak{g}, \M)$ does not depend on the Hom-section $s$. Suppose $\overline{s} : \mathfrak{g} \rightarrow \L$ is another Hom-section of $\pi$ and let $\theta_{\mathcal{E}, \overline{s}, q}$ be the corresponding $3$-cocycle defined by using $\overline{s}$ instead of $s$. Then there exists a linear map $h : \mathfrak{g} \rightarrow \V$ with $s - \overline{s} = \partial h$. Observe that
\begin{eqnarray*}
&&[\alpha_{\L}^{2}s(x), g(y,z)] - [\alpha_{\L}^{2}\overline{s}(x) , \overline{g} (y,z)]\\
&=& [\alpha_{\L}^{2}s(x) - \alpha_{\L}^{2}\overline{s}(x), g (y,z)] + [ \alpha_{\L}^{2}\overline{s}(x), (g - \overline{g})(y,z)] \\
&=& [\alpha_{\L}^{2} \partial h (x) , q ([s(y), s(z)] - s[y,z]) ] + [ \alpha_{\L}^{2}\overline{s}(x), (g - \overline{g})(y,z)] \\
&=&  [\partial \alpha_{\V}^{2}h(x),q [s(y), s(z)] - s[y,z] ] + [\alpha_{\L}^{2} \overline{s}(x), (g - \overline{g})(y,z)]\\
&=&  [ \alpha_{\V}h(x),\alpha_{\L} [s(y), s(z)] - s[y,z] ] + [\alpha_{\L}^{2} \overline{s}(x), (g - \overline{g})(y,z)] \ \ by\ \  (\alpha\text-CM2).
\end{eqnarray*}
Therefore,
\begin{eqnarray*}
\notag&&(\theta_{\mathcal{E}, s, q} - \theta_{\mathcal{E}, \overline{s}, q}) (x,y,z) \\
\notag&= & [ \alpha_{\V}h(x),\alpha_{\L}([s(y), s(z)] - s[y,z]) ] - [ \alpha_{\V}h(y),\alpha_{\L}([s(x), s(z)] - s[x,z]) ]  \\
\notag&&- [ \alpha_{\L}([s(x), s(y)] - s[x,y]),\alpha_{\V} h(z) ]  + [ \alpha_{\L}^{2}\overline{s}(x), (g - \overline{g})(y,z)]   \\
\notag&& - [ \alpha_{\L}^{2}\overline{s}(y), (g - \overline{g})(x,z)] + [ (g - \overline{g})(x,y),\alpha_{\L}^{2}\overline{s}(z)]  \\
\label{eq0001}&&+ (g - \overline{g}) (\alpha_{\L}(x),[y,z]  ) - (g - \overline{g}) ( \alpha_{\L}(y),[x,z]) - (g - \overline{g}) ([x,y],\alpha_{\L}(z)).
\end{eqnarray*}
Define a map $b : \wedge^2 \mathfrak{g} \rightarrow \V$ by
\begin{align*}
b (x,y)
&=[\alpha_{\L}s(x), h(y)]+ [h(x),\alpha_{\L}s(y) ] - h([x,y]) - [h(x), h(y)].
\end{align*}
Then a easy calculation shows that  $\partial b = \partial (g - \overline{g})$.
\begin{eqnarray*}
 &&\partial (g - \overline{g})(x,y)\\
&=&\partial q([s(x),s(y)]-s[x,y])-\partial q([\overline s(x),\overline s(y)]-\overline s[x,y])\\
&=&[s(x),s(y)]-[\overline s(x),\overline s(y)]-(s[x,y]-\overline s[x,y])\\
&=&[s(x),s(y)-\overline s(y)]+[s(x)-\overline s(x),\overline s(y)]-\partial h[x,y]\\
&=&[s(x),s(y)-\overline s(y)]+[s(x)-\overline s(x), s(y)-\partial h(y)]-\partial h[x,y]\\
&=&[s(x),\partial h(y)]+[\partial h(x),s(y)]-[\partial h(x),\partial h(y)]-\partial h[x,y]\\
&=&\partial[\alpha_{\L}s(x), h(y)]+\partial[ h(x),\alpha_{\L}s(y)]-\partial[h(x), h(y)]-\partial h[x,y] \ \ by\ \  (\alpha\text-CM1)\\
&=&\partial b(x,y).
\end{eqnarray*}

Hence $(g - \overline{g} - b) : \wedge^2 \mathfrak{g} \rightarrow \M=\ker \partial$. Thus, we have
\begin{align*}
 &(\theta_{\mathcal{E}, s, q} - \theta_{\mathcal{E}, \overline{s}, q}) (x,y,z) \\
= ~& [\alpha_{\V} h(x),\alpha_{\L}([s(y), s(z)] - s[y,z]) ] - [\alpha_{\V} h(y),\alpha_{\L}([s(x), s(z)] - s[x,z]) ]  \\
~&- [ \alpha_{\L}([s(x), s(y)] - s[x,y]),\alpha_{\V} h(z) ]  + [ \alpha_{\L}^{2}\overline{s}(x), b(y,z)]   \\
~& - [ \alpha_{\L}^{2}\overline{s}(y), b(x,z)] + [ b(x,y),\alpha_{\L}^{2}\overline{s}(z)]  \\
~&+ b (\alpha_{\L}(x),[y,z]  ) - b ( \alpha_{\L}(y),[x,z]) - b ([x,y],\alpha_{\L}(z)) \\
~& + (\delta (g - \overline{g} - b)) (x,y,z).
\end{align*}

In the right hand side of the above equation, we substitute the definition of $b$ in last six terms. After many cancellations on the right hand side, we have
\begin{align*}
~& [\alpha_{\V} h(x),\alpha_{\L}([s(y), s(z)] - s[y,z]) ] - [\alpha_{\V} h(y),\alpha_{\L}([s(x), s(z)] - s[x,z]) ] & \\
~&- [ \alpha_{\L}([s(x), s(y)] - s[x,y]),\alpha_{\V} h(z) ]  + [ \alpha_{\L}^{2}\overline{s}(x), b(y,z)]  & \\
~& - [ \alpha_{\L}^{2}\overline{s}(y), b(x,z)] + [ b(x,y),\alpha_{\L}^{2}\overline{s}(z)]  & \\
~&+ b (\alpha_{\L}(x),[y,z]  ) - b ( \alpha_{\L}(y),[x,z]) - b ([x,y],\alpha_{\L}(z)) & \\
=~&[\alpha_{\V} h(x),\alpha_{\L}[s(y), s(z)]]-[\alpha_{\V}h(x),\alpha_{\L}s[y,z]]& \\
~&+[\alpha_{\L}^{2}s(x),b(y,z)]-[\alpha_{\V}h(x),b(y,z)]+b(\alpha_{\L}x,[y,z])& \\
~&-[\alpha_{\V} h(y),\alpha_{\L}[s(x), s(z)]]+[\alpha_{\V}h(y),\alpha_{\L}s[x,z]] & \\
~&-[\alpha_{\L}^{2}s(y),b(x,z)]+[\alpha_{\V}h(y),b(x,z)]-b(\alpha_{\L}y,[x,z]) & \\
~&-[\alpha_{\L}[s(x), s(y)] ,\alpha_{\V}h(z )]+[\alpha_{\L}s[x,y],\alpha_{\V}h(z )] & \\
~&-[b(x,y),\alpha_{\L}^{2}s(z )]+[b(x,y),\alpha_{\V}h(z )]-b([x,y],\alpha_{\L}z ) & \\
=~&[\alpha_{\V} h(x),\alpha_{\L}[s(y), s(z)]] - [\alpha_{\V} h(y),\alpha_{\L}[s(x), s(z)]] - [\alpha_{\L}[s(x), s(y)] ,\alpha_{\V}h(z )] & (c1)\\
~&-[\alpha_{\V}h(x),\alpha_{\L}s[y,z]]        + [\alpha_{\V}h(y),\alpha_{\L}s[x,z]]        + [\alpha_{\L}s[x,y],\alpha_{\V}h(z )] & (c2)\\
~&+[\alpha_{\L}^{2}s(x),[\alpha_{\L}s(y),h(z)]]   -[\alpha_{\L}^{2}s(y),[\alpha_{\L}s(x),h(z)]]    -[[\alpha_{\L}s(x),h(y),\alpha_{\L}^{2}s(z)]] & (c3)\\
~&+[\alpha_{\L}^{2}s(x),[h(y),\alpha_{\L}s(z)]]   -[\alpha_{\L}^{2}s(y),[h(x),\alpha_{\L}s(z)]]    -[[h(x),\alpha_{\L}s(y)],\alpha_{\L}^{2}s(z)] &  (c4)\\
~&-[\alpha_{\L}^{2}s(x),h[y,z]]                  +[\alpha_{\L}^{2}s(y),h[x,z]]                     +[h[x,y],\alpha_{\L}^{2}s(z)] & (c5)\\
~&-[\alpha_{\L}^{2}s(x),[h(y),h(z)]]             +[\alpha_{\L}^{2}s(y),[h(x),h(z)]]              +[[h(x),h(y),\alpha_{\L}^{2}s(z)]] &  (c6)\\
~&-[\alpha_{\V}h(x),[\alpha_{\L}s(y),h(z)]]       +[\alpha_{\V}h(y),[\alpha_{\L}s(x),h(z)]]        +[[\alpha_{\L}s(x),h(y)],\alpha_{\V}h(z)] & (c7)\\
~&-[\alpha_{\V}h(x),[h(y),\alpha_{\L}s(z)]]       +[\alpha_{\V}h(y),[h(x),\alpha_{\L}s(z)]]        +[[h(x),\alpha_{\L}s(y)],\alpha_{\V}h(z)] & (c8)\\
~&+[\alpha_{\V}h(x),h[y,z]]                      -[\alpha_{\V}h(y),h[x,z]]                       -[h[x,y],\alpha_{\V}h(z)] &  (c9)\\
~&+[\alpha_{\V}h(x),[h(y),h(z)]]                 -[\alpha_{\V}h(y),[h(x),h(z)]]                  -[[h(x),h(y)],\alpha_{\V}h(z)] &  (c10)\\
~&+[\alpha_{\L}^{2}s(x),h[y,z]]                  -[\alpha_{\L}^{2}s(y),h[x,z]]                   -[\alpha_{\L}s[x,y],\alpha_{\V}h(z)] &   (c11)\\
~&+[\alpha_{\V}h(x),\alpha_{\L}s[y,z]]            -[\alpha_{\V}h(y),\alpha_{\L}s[x,z]]             -[h[x,y],\alpha_{\L}^{2}s(z)] &   (c12)\\
~&-h[\alpha_{\L}x,[y,z]]                         +h[\alpha_{\L}y,[x,z]]                          +h[[x,y],\alpha_{\L}z] &  (c13)\\
~& -[\alpha_{\V}h(x),h[y,z]]                     +[\alpha_{\V}h(y),h[x,z]]                       +[h[x,y],\alpha_{\V}h(z)]. & (c14)
\end{align*}
By direct calculation, we have
\begin{align*}
~&(c9)+(c14)\\
=~&[\alpha_{\V}h(x),h[y,z]]                    -[\alpha_{\V}h(y),h[x,z]]                     -[h[x,y],\alpha_{\V}h(z)]\\
~& -[\alpha_{\V}h(x),h[y,z]]                     +[\alpha_{\V}h(y),h[x,z]]                       +[h[x,y],\alpha_{\V}h(z)]\\
=~&0
\end{align*}
and
\begin{align*}
~&(c2)+(c5)+(c11)+(c12)\\
=~&-[\alpha_{\V}h(x),\alpha_{\L}s[y,z]]        + [\alpha_{\V}h(y),\alpha_{\L}s[x,z]]        + [\alpha_{\L}s[x,y],\alpha_{\V}h(z )]\\
~&-[\alpha_{\L}^{2}s(x),h[y,z]]                +[\alpha_{\L}^{2}s(y),h[x,z]]                 +[h[x,y],\alpha_{\L}^{2}s(z)]\\
~&+[\alpha_{\L}^{2}s(x),h[y,z]]                  -[\alpha_{\L}^{2}s(y),h[x,z]]                   -[\alpha_{\L}s[x,y],\alpha_{\V}h(z)]\\
~&+[\alpha_{\V}h(x),\alpha_{\L}s[y,z]]            -[\alpha_{\V}h(y),\alpha_{\L}s[x,z]]             -[h[x,y],\alpha_{\L}^{2}s(z)]\\
=~&0.
\end{align*}
By Hom-Leibniz identity, we have
\begin{align*}
 (c10)=~&[\alpha_{\V}h(x),[h(y),h(z)]]                 -[\alpha_{\V}h(y),[h(x),h(z)]]                  -[[h(x),h(y)],\alpha_{\V}h(z)]=0;\\
(c13) =~&-h[\alpha_{\L}x,[y,z]]                         +h[\alpha_{\L}y,[x,z]]                          +h[[x,y],\alpha_{\L}z]=0.
 \end{align*}
By (A11)-(A13), we have
\begin{align*}
~&(c1)+(c3)+(c4)\\
=~&[\alpha_{\V} h(x),\alpha_{\L}[s(y), s(z)]] - [\alpha_{\V} h(y),\alpha_{\L}[s(x), s(z)]] - [\alpha_{\L}[s(x), s(y)] ,\alpha_{\V}h(z )]\\
~&+[\alpha_{\L}^{2}s(x),[\alpha_{\L}s(y),h(z)]]   -[\alpha_{\L}^{2}s(y),[\alpha_{\L}s(x),h(z)]]    -[[\alpha_{\L}s(x),h(y),\alpha_{\L}^{2}s(z)]]\\
~& +[\alpha_{\L}^{2}s(x),[h(y),\alpha_{\L}s(z)]]   -[\alpha_{\L}^{2}s(y),[h(x),\alpha_{\L}s(z)]]    -[[h(x),\alpha_{\L}s(y)],\alpha_{\L}^{2}s(z)]\\
=~&0.
\end{align*}
Similarly, we have
\begin{align*}
~&(c6)+(c7)+(c8)\\
=~&-[\alpha_{\L}^{2}s(x),[h(y),h(z)]]             +[\alpha_{\L}^{2}s(y),[h(x),h(z)]]              +[[h(x),h(y),\alpha_{\L}^{2}s(z)]]\\
~&-[\alpha_{\V}h(x),[\alpha_{\L}s(y),h(z)]]       +[\alpha_{\V}h(y),[\alpha_{\L}s(x),h(z)]]        +[[\alpha_{\L}s(x),h(y)],\alpha_{\V}h(z)]\\
~&-[\alpha_{\V}h(x),[h(y),\alpha_{\L}s(z)]]       +[\alpha_{\V}h(y),[h(x),\alpha_{\L}s(z)]]        +[[h(x),\alpha_{\L}s(y)],\alpha_{\V}h(z)]\\
=~&0.
\end{align*}
Thus we are only left with the term $(\delta (g - \overline{g} - b)) (x,y,z)$. Hence the class of $\theta_{\mathcal{E}, s, q}$ does not depend on the Hom-section $s$.

Next consider a map
\[
\xymatrix{
\mathcal{E}:= \quad  0 \ar[r] & \M \ar[r]^{i} \ar@{=}[d] & \V \ar[r]^{\partial} \ar[d]_{\lambda} & \L \ar[r]^{\pi} \ar[d]^{\mu} & \mathfrak{g} \ar[r] \ar@{=}[d] & 0 \\
\mathcal{E}' := \quad  0 \ar[r] & \M \ar[r]_{i'} & \V' \ar[r]_{\partial'} & \L' \ar[r]_{\pi'} & \mathfrak{g} \ar[r] & 0
}
\]
of $\alpha$-crossed modules. Let $s' : \mathfrak{g} \rightarrow L'$ and $q' : \text{Im} (\partial') \rightarrow V'$ be Hom-sections of $\pi'$ and $\partial'$, respectively.
Note that $(\pi' \mu s)(x) = (\pi s) (x) = x$, for all $x \in \mathfrak{g}$. Therefore, $\mu s : \mathfrak{g} \rightarrow L'$ is another Hom-section of $\pi'$. Thus, we have
\begin{align*}
&(\theta_{\mathcal{E}, s, q} - \theta_{\mathcal{E}', \mu s, q'} ) (x,y,z) \\
= ~&[\alpha_{\L}^{2}s(x), g(y,z)] - [\alpha_{\L}^{2}s(y), g(x,z)] - [g (x,y),\alpha_{\L}^{2}s(z)]\\
~& + g (\alpha_{\L}(x),[y, z]) - g (\alpha_{\L}(y),[x,z] ) - g ([x,y],\alpha_{\L}(z))\\
~&-[\alpha_{\L'}^{2}\mu s(x), g'(y,z)] + [\alpha_{\L'}^{2}\mu s(y), g'(x,z)] + [g' (x,y),\alpha_{\L}^{2}\mu s(z)]\\
~& - g' (\alpha_{\L}(x),[y, z]) + g' (\alpha_{\L}(y),[x,z] ) + g' ([x,y],\alpha_{\L}(z)), \nonumber
\end{align*}
where $g'(x,y) = q' ([\mu s (x), \mu s (y)] - \mu s [x,y])$. Here we have used the same notation $[\cdot,\cdot]$ to denote the action of $\L$ on $\V$ and the action on $\L'$ on $\V'$. Hence, we have
\begin{align*}
&(\theta_{\mathcal{E}, s, q} - \theta_{\mathcal{E}', \mu s, q'} ) (x,y,z) \\
=~& [\alpha_{\L'}^{2}\mu s (x), (\lambda q - q' \mu) ([s(y), s(z)] - s[y,z]) ]\\
~&- [\alpha_{\L'}^{2}\mu s (y), (\lambda q - q' \mu) ([s(x), s(z)] - s[x,z]) ] \\
~&- [(\lambda q - q' \mu) ([s(x), s(y)] - s[x,y]), \alpha_{\L'}^{2}\mu s (z) ]\\
~&+ (\lambda q - q' \mu) \big(  [ \alpha_{\L} s(x), s[y,z] ] - s[\alpha_{\L} x,[y,z]]  \big)\\
~&- (\lambda q - q' \mu) \big(  [ \alpha_{\L}s(y),   s[x,z]] - s[\alpha_{\L}y,[x,z]]  \big)\\
~&- (\lambda q - q' \mu) \big(  [  s[x,y], \alpha_{\L}s(z)] - s[[x,y],\alpha_{\L}z]  \big) .
\end{align*}
It follows from the above expression that $(\theta_{\mathcal{E}, s, q} - \theta_{\mathcal{E}', \mu s, q'} ) (x,y,z) = (\delta \phi)(x,y,z)$, where $\phi : \wedge^2 \mathfrak{g} \rightarrow M$ is defined by
\begin{eqnarray*}
\phi (x,y)=(\lambda q - q' \mu)([s(x), s(y)]-s ([x, y])).
\end{eqnarray*}
Hence, $[\theta_{\mathcal{E}, s, q}] = [\theta_{\mathcal{E}', \mu s, q'}]$ in $H^3  (\mathfrak{g}, M)$. Moreover, from the first part, we have $[\theta_{\mathcal{E}', \mu s, q'}] = [\theta_{\mathcal{E}',  s', q'}]$. Hence the class $[\theta_{\mathcal{E}, s, q}]$ does not depend on the Hom-sections $s$ and $q$. We denote the corresponding class by $[\theta_\mathcal{E}]$. Therefore, the map
\begin{align*}
 \psi : \text{Cross} (\mathfrak{g}, M) \rightarrow H^3  (\mathfrak{g}, M), ~ \mathcal{E} \rightarrow [\theta_\mathcal{E}]
\end{align*}
is well-defined.
\qed

One would like to ask whether the above map in Theorem \ref{last-thm} is bijective.
But even in the Hom-Lie algebra case, it is not easy to prove this map is a bijective map.
One of the obstacle is that we don't know whether the third cohomology group is trivial or not.
Thus this question will be left for further investigated.

\subsection*{Acknowledgements}
This research was supported by NSFC (11601219, 2016BAB211003).


\end{document}